\documentclass[12pt]{amsart}
\usepackage{amsmath, amscd, mathrsfs}
\usepackage{amsthm}
\usepackage{verbatim}
\usepackage{amssymb}
 \usepackage{graphicx}
\usepackage[all, cmtip]{xy}
\usepackage[colorlinks=true,pagebackref]{hyperref}
\usepackage[dvipsnames,usenames]{color}
\hypersetup{linkcolor=RawSienna,anchorcolor=BurntOrange,citecolor=OliveGreen,filecolor=BlueViolet,menucolor=Yellow,urlcolor=OliveGreen}
\date{\today}

\newcommand{\vbar}{\overline{v}}
\newcommand{\myaddress}[1]{\indent {\sc #1}\par}
\newcommand{\myemail}[1]{\indent \emph{E-mail:} {\tt #1}}
\newcommand{\bX}{{\mathbf X}}
\newcommand{\R}{{\mathbb R}}

\newcommand{\Z}{{\mathbb Z}}
\newcommand{\C}{{\mathbb C}}
\newcommand{\G}{{\mathbb G}}
\newcommand{\T}{{\mathbb T}}
\newcommand{\Q}{{\mathbb{Q}}}
\newcommand{\I}{{\mathbb{I}}}
\newcommand{\A}{{\mathbb{A}}}
\newcommand{\PP}{{\mathbb{P}}}
\newcommand{\LL}{{\mathcal L}}
\newcommand{\cX}{{\mathcal X}}
\newcommand{\ix}{\cX}

\newcommand{\sX}{{\mathscr{X}}}
\newcommand{\one}{{\mathbf 1}}
\newcommand{\stSig}{{\boldsymbol{\Sigma}}}
\newcommand{\stsig}{{\boldsymbol{\sigma}}}

\newcommand{\dfr}{\displaystyle\frac}
\newcommand{\dsp}{\displaystyle}

\DeclareMathOperator{\Cir}{Cir}

\DeclareMathOperator{\eu}{eu}
\DeclareMathOperator{\BR}{BR}
\DeclareMathOperator{\CR}{CR}
\DeclareMathOperator{\SR}{SR}
\DeclareMathOperator{\Hom}{Hom}
\DeclareMathOperator{\Log}{Log}
\DeclareMathOperator{\Pic}{Pic}
\DeclareMathOperator{\Lie}{Lie}
\DeclareMathOperator{\Spec}{Spec}
\DeclareMathOperator{\Boxx}{Box}
\DeclareMathOperator{\Starr}{Star}
\DeclareMathOperator{\Tw}{Tw}
\DeclareMathOperator{\rank}{rank}
\newcommand{\bite}{\begin{itemize}}
\newcommand{\eite}{\end{itemize}}
\newcommand{\benu}{\begin{enumerate}}
\newcommand{\eenu}{\end{enumerate}}
\newcommand{\st}{\,|\,}
\newcommand{\bv}[1]{\text{\textbf{#1}}}
\newcommand{\vr}[1]{\left<#1\right>}
\newcommand{\modd}[1]{\mbox{ (mod }#1\mbox{)}}

\newcommand{\beq}{\begin{equation*}}
\newcommand{\eeq}{\end{equation*}}

\newcommand{\ol}{\overline}

\let\geq\geqslant
\let\leq\leqslant

\allowdisplaybreaks

\makeatletter
\def\theequation{\@arabic\c@equation}

\numberwithin{equation}{section}

\renewcommand{\det}{\operatorname{det}}

\theoremstyle{plain}
\newtheorem{theorem}{Theorem}[section]

\newtheorem{lemma}[theorem]{Lemma}
\newtheorem{corollary}[theorem]{Corollary}
\newtheorem{proposition}[theorem]{Proposition}

\theoremstyle{definition}
\newtheorem{definition}[theorem]{Definition}

\newtheorem{example}[theorem]{Example}

\newtheorem{remark}[theorem]{Remark}

\begin{document}
\title[Inertrial Chow rings of toric stacks]{Inertial Chow rings of toric stacks}
\author{Thomas Coleman and Dan Edidin}
\thanks{Research of the second author partially supported by Simons Collaboration Grant 315460}
\begin{abstract}
  For any vector bundle $V$ on a toric Deligne-Mumford stack $\ix$ the
  formalism of \cite{EJK:16} defines two intertial products
  $\star_{V^{+}}$ and $\star_{V^{-}}$ on the Chow group  
of the inertia stack. We give an explicit presentation for the integral 
$\star_{V^+}$ and
  $\star_{V^-}$ Chow rings, extending earlier work
  of Boris-Chen-Smith \cite{BCS:05} and Jiang-Tsen \cite{JiTs:10} in
  the orbifold Chow ring case, which corresponds to $V = 0$.

We also describe an {\em asymptotic} product on the rational Chow group of the inertia stack obtained by letting the rank of the bundle $V$ go to infinity.

\end{abstract}
\maketitle
\section{Introduction}
In their landmark paper, Borisov, Chen and Smith \cite{BCS:05} defined
a {\em toric Deligne-Mumford (DM) stack} $\ix(\stSig)$ 
associated to a stacky fan $\stSig= (N, \Sigma, \beta)$. Here $N$ is a finitely generated abelian group,
$\Sigma$ is a simplicial fan in $N\otimes _\Z \Q$ with $n$-rays and
$\beta \colon \Z^n \to N$ is a homormorphism such that the image of
the standard basis of $\Z^n$ generates the rays of $\Sigma$. 
Having defined toric DM stacks, the authors gave a beautiful
presentation of the orbifold Chow ring of a toric DM stack $\ix(\stSig)$. 
The orbifold Chow ring was defined by Abramovich, Graber and Vistoli \cite{AGV:08} and is an algebraic analogue
of Chen-Ruan orbifold cohomology defined in \cite{ChRu:04}. 
Precisely, they prove that if $\ix(\stSig)$ is a toric DM stack with projective
coarse moduli space $X(\Sigma)$, the rational orbifold Chow ring is a quotient of the {\em deformed group algebra}
$\Q[N]^{\stSig}$ by a linear ideal.
In subsequent work, Jiang and Tseng \cite{JiTs:10} used Iwanari's calculation 
\cite{Iwa:09} of
the integral Chow ring of a toric DM stack to give a presentation for the integral orbifold Chow ring.

In this paper we turn our attention to describing
general {\em intertial products} on toric DM
stacks. The formalism of inertial products was introducted in
\cite{EJK:10} and further developed in \cite{EJK:16}. In the latter paper the
authors show that if $\ix$ is a smooth DM stack then any vector bundle
$V$ on $\ix$ defines two inertial products $\star_{V^+}$ and
$\star_{V^-}$ on the Chow groups of the inertia stack $I\ix$. When $V$
is trival the two products agree with the usual orbifold product. When
when $V = \T$ is the tangent bundle then $\star_{\T^-}$ is the
virtual product defined in \cite{GLSUX:07}.  Our main result is an
explicit presentation for the $\star_{V^+}$ and $\star_{V^{-}}$
inertial products for large class of vector bundles on a toric DM
stack $\ix(\stSig)$.

To do this we prove the following general result about inertial Chow rings 
 of toric DM stacks.
	
\begin{theorem} \label{thm.main}
  Let $\cX(\stSig)$ be a toric Deligne-Mumford stack 
associated to a stacky fan $\stSig = (N, \Sigma, \beta)$ such that
$\Sigma$ spans $N_\R$ and the image of $\beta$ generates $N_{tors}$.
Then if $\star$ is an
  inertial product on $\cX(\stSig)$, then there is an isomorphism of
  graded rings 
\begin{equation} A^*(\cX(\stSig),\star,\Z)\cong
  \dfr{R_\stSig}{\CR(\stSig)+\BR(\star,\stSig)}.
 \end{equation}
where $R_\stSig$ is an algebra over the {\em stacky
  Stanley-Reisner} ring $S_{\stSig}$ whose generators correspond to
the twisted sectors of $\cX(\stSig)$. The ring $R_\stSig$ and the
ideal $\CR(\stSig)$ can be written explicitly in terms of the
combinatorics of the stacky fan, but the box relations ideal
$\BR(\star,\stSig)$ depends also on the choice of inertial product
$\star$.
\end{theorem}
In Proposition \ref{prop:br-v-plus-minus} we give an explicit description
of the ideals $\BR(\star_{V^{-}},\stSig)$ and $\BR(\star_{V^{+}},\stSig)$
when $V = \sum a_i {\mathcal L}_i$ where ${\mathcal L}_i$ are standard line bundles on $\ix(\stSig)$. As a corollary we obtain an explicit presentation for the virtual product.

Taking a limit over certain inertial products allows us to define (Theorem \ref{thm:asymptotic-vanishing}) two new associative products on the rational Chow groups of
$I\ix$ which we call the $\star_{+\infty}$ and $\star_{-\infty}$ products.

\begin{remark}
This article is based in large part on the first author's PhD thesis \cite{Col:16}
\end{remark}

\subsection{Outline of the paper}
In Section \ref{sec:equivChow} we begin by recalling the equivariant
intersection theory of \cite{EdGr:98}. We then show how the
equivariant projection formula can be used to compute the equivariant
Chow rings of stacks of the form $[Z/G]$ where $Z$ is the complement
of a union of a linear subspaces in a representation of a
diagonalizable group $G$ (Propositions \ref{prop.linearcomp.faithful},
\ref{prop.linearcomp.isogeny}). Since all toric stacks are naturally
expressed as quotients of this form, these results play an important
role in our calculations.

In Section \ref{sec:inertialproducts} we recall the formalism of
inertial products of \cite{EJK:10, EJK:16}. We also recall the
logarithmic trace map which is used to define the $\star_{V^+}$ and
$\star_{V^-}$ products. In Section \ref{ssec:fans,stacks} we recall
the construction of toric DM stacks and use the results of Section
\ref{sec:equivChow} to give a presentation for the ordinary Chow ring
of a toric DM stack (Propositions \ref{prop:Iwa} and \ref{prop:JiangTseng})
thereby giving simpler proofs of earlier results of Iwanari
\cite{Iwa:09} and Jiang-Tseng \cite{JiTs:10}.

In Section \ref{sec:toricinertia} we recall Borisov-Chen-Smith's description of the inertia of a toric DM stack $\ix(\stSig)$ in terms of the {\em box}
of the stacky fan $\stSig$. In the next section compute the logarithmic traces of the standard line bundles ${\mathcal L}_i$ in terms of the box.

Having laid the necessary foundation we state and prove our main results in the final two sections. In Section \ref{ssec:chow-results} we give our presentations for inertial Chow rings and in Section \ref{sec:asymptotic} we describe our new asymptotic product.

\section{Background material on equivariant Chow groups} \label{sec:equivChow}
Equivariant Chow groups were defined in \cite{EdGr:98}. We briefly
recall their definition and refer the reader to \cite{EdGr:98} for
more details.  If $G$ is a linear algebraic group acting on a scheme
or, more generally, algebraic space, we denote by $A^i_G(X)$ the
``codimension-$i$''equivariant Chow group. It is defined as the
codimension-$i$ Chow group of the space $X \times_G U$ where $U$ is an
open set in a representation $V$ on which $G$ acts freely and such
that $V \smallsetminus U$ has codimension more than $i$. It is proved
in \cite{EdGr:98} that this definition is independent of the choice
representation $V$ or open set $U$. Moreover, the equivariant
Chow group $A^k_G(X)$ is an invariant of the quotient stack $[X/G]$
and we use this as our definition of Chow groups of quotient stacks.

When $X$ is smooth, the spaces $X \times_G U$ are also smooth and we
can define an intersection product on equivariant Chow groups. In this
case we obtain a graded ring $A^*_G(X) = \oplus_{k =0}^\infty A^k_G(X)$
which we call the {\em equivariant Chow ring of $X$}. 

By their definition, equivariant Chow groups enjoy the same formal
properties as ordinary Chow groups. For example, there are
pushforwards for proper equivariant morphisms and pullbacks for flat
and lci morphisms, which are related by a projection
formula. Equivariant vector bundles define Chern classes which are
operations on the equivariant Chow groups.

An important formula for
this paper is the equivariant self-intersection formula, which states
that if $Y \stackrel{i} \hookrightarrow X$ is a regular embedding and
$\alpha \in A^*_G(Y)$ then $$i^* i_*\alpha = c_{top}(N_{Y/X}) \cap
\alpha.$$

If $V$ is an $r$-dimensional
representation of $G$ then $V$ defines an equivariant vector
bundle over a point and so we have corresponding Chern class $c_1(V), \ldots , c_r(V) \in A^*_G(pt)$. If $X$ is any $G$-space then we can pull these Chern classes back and obtain operations on $A^*_G(X)$. By abuse of notation we will also denote them as $c_i(V)$.

If $V$ is a representation of $G$ and $W \subset V$ is a $G$-invariant 
linear subspace then the normal bundle  to $W$ in $V$ 
is the bundle $W \times V/W$. In this case the projection formula reads
$$i^* i_* \alpha = c_{top}(V/W) \cap \alpha.$$

\subsection{Equivariant Chow groups for actions of diagonalizable groups}
All toric stacks have presentations as $[Z/G]$ where $G$ is a
diagonalizable group, and $Z$ is the complement of the union of a
finite collection $L_1, \ldots L_m$ of $G$-invariant linear 
subspaces in a representation $V$ of
$G$. In this section state and prove some propositions which give a
method for computing $A^*_G(Z)$ in terms of the
representation-theoretic data.
	
Precisely, let $\bX = {\bf X}(G)$ be the character group of $G$
and $A^*_G = \Z[\bX]$ be the algebra generated by this group.
By \cite{EdGr:98} we can identify this with the equivariant 
Chow ring $A^*_G(pt)$. 

\begin{proposition} \label{prop.linearcomp}
If each $L_i$ has codimension $k_i$ then
$$A^*_G(Z) = A^*_G/ \left(c_{k_1}(V/L_1), \ldots c_{k_m}(V/L_m)\right).$$
\end{proposition}

\begin{proof}
Let $L$ be  single linear subspace. Since $L$ and $V$ are $G$-vector bundles
over a point, the homotopy property of equivariant Chow groups implies
that the pullbacks $A^*_G \to A^*_GL$ and $A^*_G \to A^*_G V$ are isomorphisms.
Hence, if $i \colon L \hookrightarrow V$ is the inclusion then $i^*$ is also 
an isomorphism. 
Let $j \colon V \smallsetminus L \to V$ be the open inclusion.
The excision short exact sequence yields a surjection of
equivariant Chow rings $A^*_G(V) = A^*_G \twoheadrightarrow A^*_G(Z)$ 
whose kernel 
is $i_*A^*_G(L)$. Since $i^*$ is an isomorphism, the projection formula
implies that $i_*A^*_G(L) = i_*([L])$ where $[L]$ is the fundamental class.
By the self-intersection formula, $i^*i_*([L]) = c_{top}(N_LV)$. Since the normal bundle to $L$ in $V$ is just the representation $V/L$ we see that
$i^*i_*([L]) = c_{top}(V/L)$, so $A^*_G(Z) =  A^*_G/\left(c_{top}(V/L)\right)$.

The general case follows by induction.
\end{proof}

When the representation $V$ of $G$ is faithful then we can identify
$G$ as a closed subgroup of $\G_m^n$ where $n = \dim V$. The inclusion
$G \hookrightarrow \G_m^n$ induces a surjection of character groups
$\Z^n \to \bX$. If we let $x_i$ be the first Chern class of the image in $\bX$
 of the standard basis
vector $e_i$ then $A^*_G$ is generated by $x_1, \ldots , x_n$
and $A^*_G = \Z[x_1, \ldots , x_n]/C(G)$ where
$C(G)$ is the linear ideal representing the relations between the images 
of the $e_i$ in the finitely generated abelian group $\bX$.
Choose coordinates $z_1, \ldots , z_n$ such that the action of $\G_m^n$
(and thus $G$) is diagonalized with respect to them.
With this notation, a 
$G$-invariant linear subspace $L$ has ideal $(z_{i_1}, \ldots  , z_{i_k}))$
for some subset $\{i_1, \ldots , i_n\}$ of $\{1, \ldots , n\}$.
The normal bundle to the hyperplane $z_i =0$ is
$x_i = c_1(e_i)$ so 
$c_{top}(V/L) = x_{i_1} \ldots x_{i_k}$ 
and  we can restate Proposition \ref{prop.linearcomp}
as.
\begin{proposition} \label{prop.linearcomp.faithful} Let $V$ is a
  faithful representation of a diagonalizable group $G$. Let $Z = V
  \smallsetminus \bigcup \left(L_1 \cup \ldots L_m\right)$ where $L_k$
  is the linear subspace $V(z_{k_1}, \ldots ,  z_{k_{r_k}})$.
  Then $$A^*_G(Z) \simeq \Z[x_1, \ldots x_n]/(C(G) + (\{x_{k_1} \ldots
  x_{k_{r_k}}\}_{k=1,\ldots , m}\})
$$
\end{proposition}
Finally, we consider the situation where the action of $G$ is not faithful
but the morphism $G \to \G_m^n$ factors as $G \stackrel{\iota} \to G \hookrightarrow \G_m$ where $\iota$ is an isogeny with finite kernel and cokernel
and the second map is an immersion.
In this case let $\widetilde{x}_i$ be the first Chern class of
the image of $e_i$ under the composition of maps of character groups
${\bf X}(\G^n_m) 
\to \bX \to \bX$.
We obtain the following description of the equivariant Chow ring
in this case.
\begin{proposition} \label{prop.linearcomp.isogeny}
With the notation as in the previous paragraph
$$A^*_G(Z) = \Z[x_1, \ldots , x_n]/(C(G) + (\{\widetilde{x}_{k_1} \ldots \widetilde{x}_{k_{r_k}}\}_{k=1,\ldots , m}\})
$$
\end{proposition}
Now suppose that $V' \subset V$ be a subrepresentation of $G$ and let
$Z' = Z \cap V'$. Pullback on the closed embedding $Z' \stackrel{j} \hookrightarrow Z$
makes $A^*_G(Z')$ into an $A^*_G(Z)$-module. For each linear subspace $L_i$
let $E_i$ denote the quotient of normal bundles 
$\displaystyle{(V/L_i)/(V'/V'\cap L_i)}$.
\begin{proposition} \label{prop.relativecase}
The pullback $j^* \colon A^*_G(Z) \to A^*_G(Z')$ is surjective and $\ker j^*$
is the ideal generated by $\{ c_{top} (E_i)\}_{i=1}^m$.
\end{proposition}
\begin{proof}
This follows the descriptions of $A^*_G(V)$ and $A^*_G(V')$ given by Proposition \ref{prop.linearcomp}.
\end{proof}
We can combine Proposition \ref{prop.relativecase} with Propositions
\ref{prop.linearcomp.faithful} and \ref{prop.linearcomp.isogeny} to
obtain the an explicit description of $A^*_G(V')$ as a quotient of
$A^*_G(V)$.  To do so we need to introduce some notation. As above,
fix coordinate $z_1, \ldots , z_n$ on $V$ such that the action of $G$
is diagonalized with respect to the them. Given a linear subspace $L$,
let $M(L) = \{i| z_i|_L =0\}$. (In other words, $I(L) = ( \{z_i\}_{i
  \in M(L)})$.) Then $N_LV = \oplus_{z_i \in M(L)} {\mathcal O}(z_i)$
and we have the following proposition.
\begin{proposition} \label{prop.relativecase.xxx} If $G$ acts
  faithfully on $V$ then $$A^*_G(V') = A^*_G(V)/(\{\prod_{i_k \in
    M(L_k) \smallsetminus M(V')} x_{i_k}\}_{k=1, \ldots , m})$$ and if $G$ acts as
    the composition of an isogney with a faithful action then
$$A^*_G(V') = A^*_G(V)/(\{\prod_{i_k \in M(L_k) \smallsetminus M(V')} \tilde{x}_{i_k}\}_{k=1, \ldots , m})$$
\end{proposition}

\section{Inertial products}\label{sec:inertialproducts}
Let $G$ be an algebraic group acting on a scheme $X$.
We recall the formalism of inertial products of \cite{EJK:16}. 
However, since our application is to toric stacks, we assume that $G$ is diagonalizable throughout. 

\subsection{Inertia spaces}
\begin{definition}
	The \emph{inertia space} for the action of $G$ on $X$ is defined as
	\beq
	I_GX=\{(g,x)\st gx=x\}\subset G\times X,
	\eeq
	and the \emph{$l$-th higher inertia space} is the $l$-tuple fiber product over $X$
	\beq
	I_G^lX=\{(g_1,\dots,g_l,x)\st g_ix=x\text{ for }i=1,\dots,l\}\subset G^l\times X.
	\eeq
\end{definition}

\begin{definition}
	If $\sX$ is defined to be the quotient stack $[X/G]$, then the \emph{inertia stack} of $\sX$ is 
	\beq
	I\sX=[I_GX/G],
	\eeq
	and the \emph{$l$-th higher inertia stack} is
	\beq
	\mathbb{I}^l\sX=[\mathbb{I}^l_GX/G].
	\eeq
\end{definition}
	
\begin{definition}
	Let $X$ be a variety with the action of a group $G$. For any $l$-tuple $\bv g=(g_1,\dots, g_{l})\in G^{l}$, define the twisted sector 
	\beq
	X^{\bv g}=\{(g_1,\dots,g_l,x)\in\mathbb{I}^l_GX\}.
	\eeq
\end{definition}
	
Note that in the case $\bv g=\text{id}\in G$, we do not consider the sector $X^\bv g=X$ to be twisted.
	
\begin{remark}\label{rmk:IlIl+1}
	For convenience purposes, we will occasionally identify $\I^l_GX$ with the open and closed subset 
	\beq
	\{(g_1,\dots,g_{l+1},x)\st g_1g_2\cdots g_{l+1}=1\}
	\eeq
	of $\I^{l+1}_GX$.
\end{remark}
	
\begin{definition}
	Take any $(g_1,g_2,g_3,x)\in\I^2_GX$ where, as in Remark~\ref{rmk:IlIl+1}, $g_1g_2g_3=1$. Then for $i=1,2,3$, we define $e_i:\I^2_GX\to I_GX$ be the evaluation morphism $(g_1,g_2,g_3,x)\mapsto(g_i,x)$. Additionally, we define $\mu:\I^2_GX\to I_GX$ be the morphism $(g_1,g_2,g_3,x)\mapsto(g_1g_2,x)$.
\end{definition}

\begin{definition}\label{def:inertial-product}
	Given a class $c\in A^*_G(\I^2X)$, we define the \emph{inertial product with respect to }$c$ to be
	\beq
	x\star_cy:=\mu_*(e_1^*x\cdot e_2^*y\cdot c),
	\eeq
	where $x,y\in A^*_G(I_GX)$.
\end{definition}
We say that a 
$G$-equivariant vector bundle $\mathscr{R}$ on $\I^2X$ is associative 
if the inertial product
with respect to  $c=\eu(\mathscr{R})$ is associative. The basic example
of an associative vector bundle is the obstruction bundle used to define the orbifold product. This, and a plethora of other associative bundles can be constructed using the logarithmic trace construction of \cite{EJK:10, EJK:16}.

\subsection{Logarithmic trace and inertial products}
Let $V$ be a rank-$n$ vector bundle on $X$ which is
$G$-equivariant. Let $g$ be an element of a group
$G\subseteq(\C^\times)^n$ which acts trivially on $X$ and whose action
is an automorphism of the fibers of $V\to X$. Under these conditions,
the eigenbundles for the action of $g$ will be $G$-subbundles of $E$.
	
Denote the eigenvalues for the action of $g$ on $V$ by
$\exp(2\pi\sqrt{-1}\lambda_k)$ for $1\leq k\leq r$, and without loss
of generality we can say $0\leq\lambda_k<1$ for each $k$. We will use
$E_k$ to denote the eigenbundle corresponding to $\lambda_k$.

\begin{definition} \cite[Definition 4.1]{EJK:10}\label{def:logtrace}
The \emph{logarithmic trace} of $V$ for the action of
  $g$ is \beq L(g)(V) = \sum_{k=1}^r\lambda_kV_k\in K_{G}(X)\otimes\R
  \eeq .
\end{definition}

\begin{remark}
  The assumption that $G$ is diagonalizable implies that the full 
group $G$ acts
  on $X^g$ for any $g\in G$. This simplifies
  Definitions~\ref{def:logtrace} and \ref{def:logrestr} as compared to
  those found in \cite{EJK:10, EJK:16}.
\end{remark}

\begin{example}\label{ex:log-trace}
	Let $Z=\C^4-\{(0,0,0,0)\}$, and let $G=\C^\times$ act on $Z$, via
	\beq
	g\cdot(z_1,z_2,z_3,z_4)=\left(g^2z_1, g^4z_2, g^5z_3, g^6z_4\right). 
	\eeq
	Let $V$ be a vector bundle on $Z$ with subbundles $L_i$ corresponding to the divisors where $z_i=0$. That is, $V$ is the tangent bundle $T_V=L_1+L_2+L_3+L_4$.
	
	First, note that $I_GZ=\coprod Z^g$, where the union is over the twelve elements $g\in G=\C^\times$ which fix at least one point of $Z$. In order to fix a point of $Z$, $g$ must be either a second, fourth, fifth or sixth root of unity. 
	
	Let's look at $g=e^{2\pi\sqrt{-1}/3}\in G$, and compute its logarithmic trace. Note that it only makes sense to consider the action of $g$ on $V$ if we restrict to $V|_{Z^g}$. 
	
	Under the induced action of $\alpha$ on $Z$, we have
	\beq
		g\cdot \left( z_1,z_2,z_3,z_4 \right) 
		= 
		\left(e^{4\pi\sqrt{-1}/3}z_1,\, 
			e^{2\pi\sqrt{-1}/3}z_2,\, 
			e^{4\pi\sqrt{-1}/3}z_3,\, 
			z_4
		\right). 
	\eeq
	So $g$ has three eigenvalues and eigenbundles for its action on $V|_{Z^g}$: 
	\beq
	\begin{array}{ccl}
	\lambda_1 = 1 & \text{for} & E_1 = L_4\\
	\lambda_2 = e^{4\pi\sqrt{-1}/3} & \text{for} & E_2 = L_1\text{ or }E_2=L_3\\
	\lambda_3 = e^{2\pi\sqrt{-1}/3} & \text{for} & E_3 = L_2
	\end{array}
	\eeq
	Thus, the logarithmic trace of $g$ on $V|_{Z^g}$ is
	\beq
		L(g)(V|_{Z^g}) = \frac23L_1+\frac13L_2+\frac23L_3.
	\eeq
	Notice that the coefficient on $L_4$ in the logarithmic trace is 0; this is a direct consequence of the fact that $Z^g = \{(0,0,0,z_4)\st z_4\neq0\}$. 
\end{example}

\begin{definition}\cite[Definition 5.3]{EJK:10}\label{def:logrestr}
	Let $\bv g$ be an $l$-tuple $(g_1,\dots g_l)$ such that there exists a finite subgroup of $G$ containing every $g_i$ for $1\leq i\leq l$, and with the additional stipulation that $g_1\cdots g_l=1$. Then the \emph{logarithmic restriction} of $E$ is the class in $K_G(X^{\bv g})$ defined by the formula 
	\begin{equation}
	V(\bv g)=\sum_{i=1}^l L(g_i)(V|_{X^\bv g})+V^{\bv g}-V|_{X^\bv g}.
	\end{equation}
	The assignment $LR:V\mapsto V(\bv g)$ is called the \emph{logarithmic restriction} map, and takes non-negative elements $E\in K_G(X)$ to non-negative elements $LR(V)\in K_G(X^{\bv g})$.
\end{definition}

\begin{example}\label{ex:log-restr}
	We'll build on Example~\ref{ex:log-trace}, with the same $Z$, $G$, $\alpha$ and $V$; additionally, let $\bv g = \left(e^{\pi\sqrt{-1}/3}, e^{2\pi\sqrt{-1}/3}, e^{\pi\sqrt{-1}}\right)$. Then we have $Z^{\bv g} = \{(0,0,0,z_4)\st z_4\neq 0\}\subseteq Z$, which in turn gives $V^{\bv g} = L_4$. 
	
	Then the logarithmic restriction of $V$ at $\bv g$ is
	\begin{align*}
		&V(\bv g) \\
		&= L\left(e^{\pi\sqrt{-1}/3}\right)\left(V|_{Z^\bv g}\right)+L\left(e^{2\pi\sqrt{-1}/3}\right)\left(V|_{Z^\bv g}\right)+L\left(e^{\pi\sqrt{-1}}\right)\left(V|_{Z^\bv g}\right)+V^{\bv g}-V|_{Z^\bv g} \\
		&= \left(\dfr13L_1+\dfr23L_2+\dfr56L_3\right)+\left(\frac23L_1+\frac13L_2+\frac23L_3\right) + \dfr12L_3\\
		& + L_4 - (L_1+L_2+L_3+L_4)\\
		&= L_3.
	\end{align*}
\end{example}

\begin{proposition}\cite[Proposition 4.0.10]{EJK:16}\label{def:E-plus-minus}
	Let $V$ be a $G$-equivariant bundle on $X$, and let $g_1$, $g_2$ lie in a common compact subgroup of $G$. Then the virtual bundles 
	\begin{equation}
	V^+(g_1,g_2)=L(g_1)(V|_{X^{g_1,g_2}})+L(g_2)(V|_{X^{g_1,g_2}})-L(g_1g_2)(V|_{X^{g_1,g_2}})
	\end{equation}
	and
	\begin{equation}
	V^-(g_1,g_2)=L(g_1^{-1})(V|_{X^{g_1,g_2}})+L(g_2^{-1})(V|_{X^{g_1,g_2}})-L(g_1^{-1}g_2^{-1})(V|_{X^{g_1,g_2}})
	\end{equation}
	are represented by non-negative integral elements in $K_G(X^{g_1,g_2})$.
\end{proposition}

\begin{definition}\label{def:R-plus-minus}
  Let $\bv g=(g_1,g_2)\in G^2$. Define classes $R^+V$ and $R^-V$ in
  $K_G(\I^2X)$ by setting the component of $R^+V$
  (resp. $R^-V$) in $K_G((X^{\bv g}))$ to be $V^+(g_1,g_2)$
  (resp. $V^-(g_1,g_2)$).
\end{definition}

\begin{theorem}\cite[Theorem 4.0.12]{EJK:16}\label{thm:inertial-is-assoc}
  Let $\T$ be the $G$-equivariant vector bundle corresponding to
the tangent bundle of $[X/G]$ and let $\mathscr{R}^+V$ be the vector bundle $LR(\T)+R^+V$ for any
  $G$-equivariant vector bundle $V$ on $X$. Then the inertial product
  with respect to $\eu(\mathscr{R}^+V)$, which we call $\star_{V^+}$,
  is associative. Similarly, $\eu(\mathscr{R}^-V)=\eu(LR(\T)+R^-V)$,
  called $\star_{V^-}$, also defines an associative inertial product.
\end{theorem}

\begin{remark} \cite{EJK:16}\label{def:orb-prod}
	The \emph{orbifold product} $\star_{orb}$ is defined by taking $V=0$ in Definition~\ref{def:inertial-product} and Theorem~\ref{thm:inertial-is-assoc} (with either $\mathscr{R}^+$ or $\mathscr{R}^-$). 
	The \emph{virtual product} $\star_{virt}$ of \cite{GLSUX:07} 
is the $\star_{\T^-}$ product.
\end{remark}

\section{Stacky fans and toric stacks}\label{ssec:fans,stacks}

\subsection{Definitions and setup}
We recall the construction of toric DM stacks from stacky fans following \cite{BCS:05}.
	
A stacky fan $\stSig$ is a triple $(N,\Sigma,\beta)$, where $N$ is a
finitely generated abelian group of rank $d$, $\Sigma$ is a simplicial
fan with $n$ rays, $\rho_1, \ldots , \rho_n$
 in the $d$-dimensional $\Q$-vector space
$N_\Q:=N\otimes_\Z\Q$, and $\beta$ is a homomorphism from $\Z^n$ to
$N$ defined by a choice of elements
$b_1, \ldots , b_n \in N$ such that the image of $b_i$ in $N \otimes \Q$
generates the ray $\rho_i$. We call $b_1,\dots b_n$ the \emph{distinguished points} of $\stSig$ in $N$. 

Let $X_\Sigma$ be the toric variety associated to the fan
$\Sigma$. Let $T_N:=\Hom(N^*,\C^\times)$ and
$T_L:=\Hom((\Z^n)^*,\C^\times)\cong(\C^\times)^{n}$.  Then the natural
map $\beta^*:N^*\to (\Z^n)^*$ induces a morphism of tori
$T_\beta:T_L\to T_N$.

Let $N=\Z^d\oplus\Z/{m_1}\oplus\dots\oplus\Z/{m_r}$ be the invariant
factor form of $N$. We define $\beta_{aug}:\Z^n\oplus\Z^r\to\Z^{d+r}$
to be the map represented by the matrix \beq
\left[\begin{array}{cccccc}
    b_{1,1} &  & b_{n,1} & 0 & 0 & 0 \\
    & \ddots & 	 & 0 & \ddots & 0 \\
    b_{1,d}   & 	 & b_{n,d} & 0 & 0 & 0 \\
    b_{1,d+1} &  & b_{n,d+1} & m_1 & 0 & 0\\
    & \ddots &  & 0 & \ddots & 0\\
    b_{1,d+r} & & b_{n,d+r} & 0 & 0 & m_r
		\end{array}\right]
	\eeq 
	where the $i$-th column of $\beta_{aug}$ is a lift of $b_i$ under the natural surjection $\Z^d\oplus\Z^r\to N$. 
	
	For convenience, we will represent the integers $b_{i,d+j}$ (which are elements of $\Z/{m_l}$) using elements of the set $\{0,1,\dots,m_l-1\}$ for $l=1,\dots,r$, but the construction is independent of choice (cf Remark \ref{rem.choiceindep}).
	
	Define a map $E^\beta:(\C^\times)^{n+r}\to(\C^\times)^{d+r}$
        by exponentiating $\beta_{aug}$. That is,
	\begin{equation*}
		E^\beta(\gamma_1,\dots,\gamma_n,s_1,\dots,s_r) = \left(\prod_{i=1}^n {\gamma_i}^{b_{i,1}},\dots,\prod_{i=1}^n {\gamma_i}^{b_{i,d}},\,{s_1}^{m_1}\prod_{i=1}^n {\gamma_i}^{b_{i,d+1}},\dots,{s_r}^{m_r}\prod_{i=1}^n {\gamma_i}^{b_{i,d+r}} \right).
	\end{equation*}
	Let $G$ be the kernel of the map $E^\beta$. As a subgroup of
        $(\C^\times)^{n+r}$:\begin{equation}\label{eq:def-G}
          G=\left\{(\gamma_1,\dots,\gamma_n,s_1,\dots,s_r)\left|\,
			\begin{aligned} 
			&\prod_{i=1}^n {\gamma_i}^{b_{i,j}}=1\text{ for }1\leq j\leq d, \text{ and }\\ 
			{s_l}^{m_l}&\cdot\prod_{i=1}^n {\gamma_i}^{b_{i,d+l}}=1\text{ for }1\leq l\leq r 
			\end{aligned}
		\right.\right\}.
		\end{equation}

Let $\C[z_{\rho}\st \rho\in\Sigma(1)]=\C[z_1,\dots,z_n]$ be the total
coordinate ring of $X_\Sigma$. There is a natural action of $G$ on
$\A^n = \Spec \C[z_1, \ldots , z_n]$ given by \beq
(\gamma_1,\dots,\gamma_n,s_1,\dots,s_n)\cdot(z_1,\dots,z_n) =
(\gamma_1z_1,\dots,\gamma_nz_n).  \eeq Let $Z = \A^n \smallsetminus
 V(J_{\Sigma})$
where $J_{\Sigma}:=\vr{\left.\prod_{\rho_i\not\subset\sigma}z_i\right|
  \sigma\in\Sigma}$ is the irrelevant ideal.  Since $V(J_{\Sigma})$ is
a union of coordinate subspaces, $Z=\A^n\setminus V(J_{\Sigma})$ is
$G$-invariant so there is an action of $G$ on $Z$.

The \emph{toric stack associated to $\stSig$} is defined to be the
stack quotient $\cX(\stSig):=[X_\Sigma/G]\cong[Z/G]$ under this
action. With the assumptions of our setup $\cX(\stSig)$ will be a
smooth, separated Deligne-Mumford stack \cite[Proposition 3.2]{BCS:05}.

\begin{remark}\label{rmk:G}
	If $N$ is torsion-free, we can simplify this explicit construction. Since $r=0$, we get
	\begin{equation}\label{eq:def2-G}
	G=\left\{(\gamma_1,\dots,\gamma_n)\left|\, 
			\prod_{i=1}^n {\gamma_i}^{b_{i,j}}=1\text{ for }1\leq j\leq d
	\right.\right\}
	\end{equation} In this case the action of $G$ on $\A^n$ corresponds
to a faithful representation. Thus, the stabilizer of a general point is trivial
and the quotient stack $[Z/G]$ is an effective orbifold.
\end{remark}

If $N$ has torsion then $\sX(\stSig)$ then the action of $G$ on $Z$ is not 
generically free as the following example shows.
\begin{example}\label{ex:p64intro}
	Let $N=\Z\oplus\Z/2$ and let $\Sigma$ be the complete fan in $N_\Q\cong \Q$. Define $\beta$ by the distinguished points $b_1=(2,1)$ and $b_2=(-3,0)$ in $N$. This defines a stacky fan $\stSig=(N, \Sigma,\beta)$. Notice that $b_1$ and $b_2$ are \emph{not} minimal generators of $\rho_1$ and $\rho_2$ over $N_\Q$.
	
	As $N$ is not a lattice, we must compute $G$ using \eqref{eq:def-G}. We have the homomorphism $\beta_{aug}$:
		\beq
		\left[\begin{array}{rrr} 2&-3&0\\ 1&0&2 \end{array}\right]:\Z^3\to\Z^2
		\eeq
	Hence,
	\begin{align*}
	G=\ker(E^\beta)=&\left\{(\gamma_1,\gamma_2,s_1)\in(\C^\times)^3\st {\gamma_1}^2{\gamma_2}^{-3}=1, \gamma_1{s_1}^2=1 \right\}.
	\end{align*}
	
	Computing $J_\Sigma$ the usual way, we have $Z=\C^2\,\setminus\{(0,0)\}$, so the $G$-action on $Z$ is 
	\beq
	(\gamma_1,\gamma_2,s_1)\cdot(z_1,z_2)=(\gamma_1z_1,\gamma_2z_2).
	\eeq
	But since $G\cong\left\{\left.(\lambda^6,\lambda^4,\lambda^{-3})\,\right|\,\lambda\in\C^\times\right\}\cong\C^\times$, the $G$-action is equivalent to the $\C^\times$-action 
	\beq
	\lambda\cdot(z_1,z_2)=(\lambda^6z_1,\lambda^4z_2).
	\eeq
	That is, $\cX(\stSig)$ is the weighted projective stack $\PP(6,4)$, which of course is not a reduced orbifold.
\end{example}

\begin{remark} \label{rem.choiceindep}
  In Example~\ref{ex:p64intro}, we made a specific choice of
  distinguished points $b_1=(2,1)$ and $b_2=(-3,0)$, 
 If we had chosen any other equivalent lift under the
  surjection $\Z\oplus\Z\to\Z\oplus\Z/2$, we would have obtained the
  same stack. As an extreme example, we can pick $b_1=(2,87)$ and
  $b_2=(-3,54)$ which gives us \beq
  G\cong\left\{(\lambda^6,\lambda^4,\lambda^{-369})|\lambda\in\C^\times\right\},
  \eeq which yields the same stack $\cX(\stSig)=\PP(6,4)$.
\end{remark}

\subsection{Characters and Line bundles} \label{sec.charline}

Let $\cX(\stSig)=[Z/G]$ be a toric Deligne-Mumford stack
associated to the stacky fan $\stSig=(N, \Sigma, \beta)$ 
and let $X(\Sigma)$ be the underlying toric variety.

A line bundle on $[Z/G]$ is a $G$-equivariant line bundle on
$Z$. Since $Z$ is an open set in a representation of $G$, $\Pic([Z/G])
= A^1_G(Z)$ is generated by characters of $G$. In, addition if we
assume that the fan $\Sigma$ has maximal dimension (for example if
$X(\Sigma)$ is proper) then the complement of $Z$ in
$\A^{|\Sigma(1)|}$ has codimension at least two, so $\Pic([Z/G]) =
\Pic^G(\A^{|\Sigma(1)|}) = \bX(G) = N^\vee.$ When $[Z/G]$ is a reduced
orbifold, meaning that the generic stabilizer is trivial (ie the
action of $G$ is effective) then $Pic(\sX(\stSig)) = Pic(X(\Sigma)$. In
this case if we let $L_i$ be the line bundle on $X(\Sigma)$ corresponding
to the ray $\rho_i$ then $Pic(\sX(\stSig))$ is generated by the
pullback of these line bundles to $\sX(\stSig)$ which we also denote
by $L_i$. If we identify ${\bf X}(\G_m^n) = \Z^n$ and let $e_i$ be
character corresponds to the $i$-th standard basis vector, then $L_i$
corresponds to the image of $e_i$ in ${\bf X}(G)$ under the surjective
morphism $\Z^n \to {\bf X}(G)= N^\vee$.

In general the action of $G$ on $Z$ is not effective. However, if we
assume that $b_1 \ldots , b_n$ generate $N_{tors}$ then Jiang and Tseng
\cite[Section 3]{JiTs:10} prove that the action of $G$ on $\A^n$ corresponds to a
map $G \stackrel{\iota} \to G \stackrel{i} \hookrightarrow \G_m^n$
where $\iota$ is an isogeny and $i$ is an immersion.  The quotient
$[Z/_{i} G]$ where the subscript $_i$ indicate that the action is the
faithful action of $G$ produces the {\em rigidification} of
$\sX(\stSig)_{red}$ of $\sX(\stSig)$ in the sense of \cite{AGV:08}.
Let $\LL_i$ be the image of $e_i$ under the composite map $\Z^n \to
\bX(G) \stackrel{\iota^*} \to \bX(G)$. Equivalently, $\LL_i$ is the
pullback of the line bundle corresponding to the ray $\rho_i$ in the
underlying toric variety. If $z_1, \ldots , z_n$ are coordinates on
$\A^n$ then ${\mathcal O}(z_i) = \LL_i$.

\begin{lemma}\cite[Section 3.3] {JiTs:10} \label{lem:Li-gen-char-gp}
If $\beta_1, \ldots , \beta_n$ generate $N_{tors}$ then
$\Pic(\cX(\stSig))$ is generated by the line bundles $L_1, \ldots L_n$.
\end{lemma}

\begin{remark}
Let $x_i  = c_1(L_i)$ and $\tilde{x}_i = c_1(\LL_i)$, then
$$	\tilde{x}_i=\sum_{k=1}^n f_{k,i}x_i$$
for some integers $f_{k,i}$ such that $\det(f_{k,j}) \neq 0$. Jiang and Tseng refer to this equation as the {\em associated formula} for the $\tilde{x}_i$.
\end{remark}	
\begin{example}
  The stack $\cX(\stSig)=\PP(6,4)$ of Example~\ref{ex:p64intro} has
  distinguished points $b_1=(2,1)$ and $b_2=(-3,0)$ in
  $N=\Z\oplus\Z/2$. The underlying toric variety is $\PP(3,2)$, with
  line bundles $L_1$ and $L_2$ having Chern classes $x_1:=c_1(L_1)=3t$
  and $x_2:=c_1(L_2)=2t$ where $t$ is the first Chern class
of the defining character of $G =\C^\times$. 
The associated formula of the stacky fan
  $\stSig$ is \beq 
\tilde{x}_1=2x_1 = 6t 
\quad\quad \tilde{x}_2=
2x_2= 4t 
\eeq
\end{example}

\begin{example}\label{ex:noneffective}
	Now let $N=\Z\oplus\Z/2$ and let $\Sigma$ be the complete fan in $N_\Q\cong \Q$ but define $\beta$ by the distinguished points $b_1=(2,0)$ and $b_2=(-3,0)$ in $N$ so the image of $\beta_1, \beta_2$ does not generate $N_{tors}$
In this case our augmentation homorphism is 
		\beq
		\left[\begin{array}{rrr} 2&-3&0 \\ 0&0&2 \end{array}\right]:\Z^3\to\Z^2
		\eeq
	Hence,
	\begin{align*}
	G=\ker(E^\beta)=&\left\{(\gamma_1,\gamma_2,s_1)\in(\C^\times)^3\st {\gamma_1}^2{\gamma_2}^{-3}=1, {s_1}^2=1 \right\}.
	\end{align*}
so $G = \C^\times \times \mu_2$ with $G$ acting by 
$(\lambda, \mu) (z_1, z_2) = (\lambda^2 z_1, \lambda^3 z_2)$.
Thus $\cX(\stSig) = \PP(2,3) \times B\mu_2$ and $\Pic(\cX(\stSig)) = \Z \oplus \Z_2$. Here we have $L_i = {\mathcal L_i}$ but $L_1, L_2$ do not generate 
the full Picard group because they do not detect the torsion.
\end{example}

\subsection{The Chow ring of a toric stack}
Here, we use Propositions \ref{prop.linearcomp.faithful} and
 \ref{prop.linearcomp.isogeny} to give simple proofs of the results 
of \cite{Iwa:09, JiTs:10} on the integral
Chow rings of toric stacks. 
Note that, unlike \cite{JiTs:10} we do not
require that the stack be proper but we require that every maximal
cone of $\Sigma$ has dimension $d = \rank N$. In particular this
implies that the vectors $b_1, \ldots , b_n$ span $N \otimes \R$.

\begin{proposition}\cite[Theorem 2.2]{Iwa:09} \label{prop:Iwa} 
If $N$ is torsion free and $b_1, \ldots ,  b_n$ span $N \otimes \R$
then $A^*(\sX(\stSig) = \Z[x_1, \ldots x_n]/\left(I_{\stSig} + C(\stSig)\right)$.

where $I_\stSig$ is the ideal generated by monomials 
$$\{x_{i_1}, \ldots , x_{i_k}: \rho_{i_1} , \ldots ,\rho_{i_k} \text{do not lie in a cone of }\Sigma\}.$$

and 
$C(\stSig)$ is the ideal generated by the linear relations
$$\left(\sum_{i=1}^n \theta(\ol{b}_i) x_i\right)_{\theta \in N^*}$$
\end{proposition}
\begin{proof}
We use Proposition \ref{prop.linearcomp.faithful}. 

The definition of $G$ implies that $\bX(G)$ is
cokernel of the map $N^*=(\Z^d)^* \stackrel{B^*} \to (\Z^n)^*$.  Thus 
$\bX(G)$ is the quotient of $\Z^n$ by the 
linear relations $\{\sum_{i} \theta(b_i) x_i\}_{\theta \in N^*}$,
so $C(\stSig)= C(G)$ where $C(G)$ is as in Proposition \ref{prop.linearcomp.faithful}.

If we use coordinates $z_1, \ldots , z_n$ on $\A^n$ then
$V(J_{\Sigma})$ decomposes as the union of the linear subspaces
$\{L_\sigma\}=
V(z_{k_1}, \ldots , z_{k_r})$ where $\rho_{k_1}, \ldots \rho_{k_r}$
are the rays of $\Sigma(1) \smallsetminus \sigma(1)$ and $\sigma$ runs through all maximal cones. Thus, 
$I_\stSig$ is exactly the ideal generated by the products
$x_{k_1}x_{k_2} \ldots x_{k_{r}}$ as in Proposition
  \ref{prop.linearcomp.faithful}.

\end{proof}

\begin{proposition}\cite[Theorem 1.1]{JiTs:10} \label{prop:JiangTseng} 
Now let $N$ be an arbitrary finitely generated abelian group
and assume that $b_1, \ldots , b_n$ generate $N_{tors}$
Then $$A^*(\sX(\stSig) = \Z[x_1, \ldots x_n]/\left(I_{\stSig} + C(\stSig)\right).$$

where $I_\stSig$ is the ideal generated by monomials 
$$\{\widetilde{x}_{i_1}, \ldots , \widetilde{x}_{i_k}: \rho_{i_1} , \ldots ,\rho_{i_k} \text{do not lie in a cone of }\Sigma\}.$$
where $\tilde{x}_i$ is the first Chern class of the image of the character
$e_i$ in ${\mathbf X}(G)$
and 
$C(\stSig)$ is the ideal generated by the linear relations
$$\left(\sum_{i=1}^n \theta(\overline{b}_i) x_i\right)_{\theta \in N^*}$$

\end{proposition}
\begin{proof}
  The key point, proved by Jiang and Tseng, is that the assumption
  that $b_1, \ldots , b_n$ means that the representation $G \to
  \G_m^n$ factors as an isogeny $G \to G$ followed by faithful
  representation $G \hookrightarrow \G_m^n$ corresponding to the map
  $\Z^n \to \overline{N}= N/N_{tors}$. The proposition now follows
  from Proposition \ref{prop.linearcomp.isogeny}.
\end{proof}

\section{The interia of a toric stack} \label{sec:toricinertia}

\subsection{The box of a stacky fan}
Let $\stSig = (N, \Sigma, \beta)$ be a stacky fan such that $\Sigma$ spans
$N \otimes \Q$. Following \cite{BCS:05} we introduce $\Boxx(\stSig)$, which plays an important role in 
the presentations for our inertial Chow rings. 

	Let $\ol{N}:=N/N_{tor}$ and for any $v\in N$, we use $\ol{v}$ to denote the image of $v$ in the natural projection $N\to\ol N$. In a slight abuse of notation, we will sometimes refer to $\ol{v}$ as being in $\Sigma$. 
	
	Denote by $\sigma(\ol{v})$ the unique minimal cone
        $\sigma\in\Sigma$ which contains $\ol{v}$. More generally,
        $\sigma(\ol{v_1},...,\ol{v_k})$ is the minimal cone in
        $\Sigma$ containing each of $\ol{v_1},\dots,\ol{v_k}$.
	
\begin{definition}\label{def:box}
  For any cone $\sigma\in\Sigma$, we define $\Boxx(\stsig)$ to be the
  set of all elements $v\in N$ such that $\ol{v}=\sum_{i|\rho_i \in
    \sigma(\ol{v})(1)} q_i\ol{b_i}$ with $q_i \in [0,1) \cap \Q$
(Here we use $\sigma(\ol{v})(1)$ to denote the rays of the cone $\sigma(\ol{v})$.)

Define $\Boxx(\stSig)$ to be the union of $\Boxx(\stsig)$ for all
  cones $\sigma\in\Sigma$.
\end{definition}
	
For any cone $\sigma\in\Sigma$, define
\begin{equation}\label{eq:defNsigma}
N_\sigma:=\vr{b_i\st\rho_i\in\sigma}
\end{equation}
as a subgroup of $N$, and define $N(\sigma)$ be the finite quotient
group $N/N_\sigma$. The set $\Boxx(\stsig)$ has a natural bijection
with $N(\sigma)$, 
and gives a choice of coset representatives for the quotient 
group.

\begin{example}\label{ex:p64box}
  Let $\cX(\stSig)$ be the toric stack $\PP(6,4)$, constructed in
  Example~\ref{ex:p64intro}.  The squares in Figure~\ref{fig:p64}
  represent the eight elements of $\Boxx(\stSig)$. For example, since
  $\ol N=\Z$, we have that $v=(1,1)\in N$ is in $\Boxx(\stSig)$
  because $\ol{v}=\frac12\ol{b_1}\in\Boxx(\sigma_1)$, where $\sigma_1$
  is the one-dimensional cone in the positive direction on the
  horizontal axis.
	
	Notice also that $(5,1)\in N$, but $(5,1)\notin \Boxx(\stSig)$. The box element which is equivalent to $(5,1)$ in $N(\sigma_1)$ is $(1,1)$, since $(5,1)=2b_1+(1,1)\in N$.
\end{example}

 \begin{figure} \label{fig:p64box}
 	\centering
 	\includegraphics[width=3.8in]{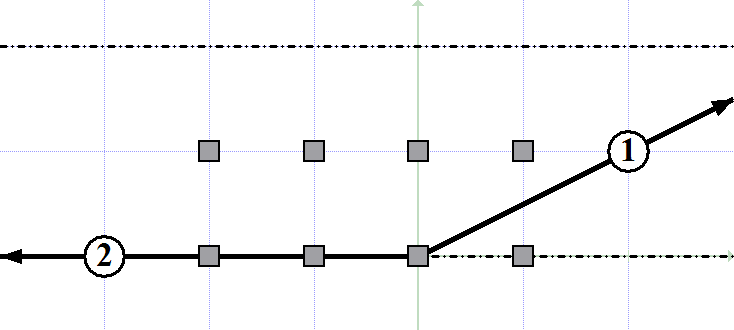}
 	\caption{The stacky fan of Example~\ref{ex:p64intro}, which yields the toric stack $\PP(6,4)$. The vertical axis here is representing $\Z/2$; that is to say, the two dotted lines are identified. The squares are elements of $\Boxx(\stSig)$.}
 	\label{fig:p64}
 \end{figure}

The set $\Boxx(\stsig)$ inherits the abelian group structure of
$N(\sigma)$. We give an example to illustrate how the addition works.
	
\begin{example}\label{ex:boxadd}
	We'll revisit the stacky fan of Example~\ref{ex:p64intro}
	.  We have $\Boxx(\sigma_1)=\{0,v_1, v_2, v_3\}$, where $0=(0,0)$, $v_1=(1,1)$, $v_2=(1,0)$ and $v_3=(0,1)$. Then we have the following addition table for $\Boxx(\boldsymbol{\sigma_1})$:
	\begin{equation}\label{eq:p64boxaddtable}
	\begin{array}{c|cccc}
		+&0&v_1&v_2&v_3\\
		\cline{1-5}
		0 &0 &v_1 &v_2 &v_3\\
		v_1 &v_1 & v_3 &0 & v_2\\
		v_2&v_2 &0 &v_3 & v_1 \\
		v_3& v_3& v_2& v_1&0
	\end{array}
	\end{equation}
\end{example}

\begin{proposition}\cite[Lemma 4.6]{BCS:05}\label{prop:box-corr}
	If $\Sigma$ is a complete fan, then the elements $v\in\Boxx(\stSig)$ are in one-to-one correspondence with elements $g\in G$ which fix a point of $Z$.
\end{proposition}
\begin{proof}
	Let $G_\sigma$ be the subgroup of $G$ defined by the equations $\gamma_i=1$ for all $\rho_i\notin\sigma$. That is, 
	\begin{equation}
	G_\sigma=\left\{(\gamma_1,\dots,\gamma_n,s_1,\dots,s_r)\left|\, 
		\begin{aligned}\label{eq:Gsig}
		&\prod_{\rho_i\in\sigma} {\gamma_i}^{b_{i,j}}=1\text{ for }1\leq j\leq d, \text{ and}\\ 
		{s_l}^{m_l}&\cdot\prod_{\rho_i\in\sigma} {\gamma_i}^{b_{i,d+l}}=1\text{ for }1\leq l\leq r,\text{ and} \\
		& \gamma_i=1\text{ for all }\rho_i\notin\sigma
		\end{aligned}
	\right.\right\}.
	\end{equation}
	The set of elements of $G$ which fix a point of $Z$ will be  $\bigcup_{\sigma\in\Sigma}G_\sigma$.
	
	We'll prove the statement for $g\in G_\sigma$ and $v\in \Boxx(\stsig)$; taking unions will complete the proof.
	
	For $1\leq j\leq d$, recall that $b_{i,j}$ denotes the $j$-th entry of the distinguished point $b_i$ corresponding to the ray $\rho_i$. We claim that for $g=(\gamma_1,\dots,\gamma_n,s_1,\dots,s_r)\in G_\sigma$, the correspondence is $g\leftrightarrow v$, where $v$ is given by
	\begin{equation}\label{eq:gvcorr}
		\left(
		\sum_{i=1}^n b_{i,1}\frac{\Log(\gamma_i)}{2\pi\sqrt{-1}},
		...,
		\sum_{i=1}^nb_{i,d}\frac{\Log(\gamma_i)}{2\pi\sqrt{-1}},
		\dfr{m_1\Log(s_1)-\Log({s_1}^{m_1})}{2\pi\sqrt{-1}},
		...,
		\dfr{m_1\Log(s_r)-\Log({s_r}^{m_r})}{2\pi\sqrt{-1}}
		\right),
	\end{equation}
	with the convention that for any $\zeta\in\C^\times$ with $|\zeta|=1$, we use the principal  branch of the natural logarithm, $0\leq\dfr{\Log \zeta}{2\pi\sqrt{-1}}<1$. 
	
	We first show that the way we defined $v$ makes it an element of $N$. First, we have
	\begin{equation}\label{eq:392403}
	\, \dfr{1}{2\pi\sqrt{-1}}\sum_{i=1}^n b_{i,j}\Log (\gamma_i)
	\equiv \, \dfr{1}{2\pi\sqrt{-1}}\sum_{i=1}^n \Log\left({\gamma_i}^{b_{i,j}}\right) 
	\equiv \,0\modd1.
	\end{equation}
	We also have $(m_l\Log(s_l)-\Log({s_l}^{m_l}))/2\pi\sqrt{-1}\in\{0,\dots,m_l-1\}=\Z/m_l$. So we have $v\in N$.
	
	Next, we show that there exist rational numbers $q_i\in[0,1)\cap\Q$ such that $\ol{v}=\sum_{\rho_i\in\sigma}q_i\ol{b_i}$. Indeed, 
	\beq
	\ol{v}= \left(
		\sum_{i=1}^n b_{i,1}\frac{\Log(\gamma_i)}{2\pi\sqrt{-1}},
		\dots,
		\sum_{i=1}^nb_{i,d}\frac{\Log(\gamma_i)}{2\pi\sqrt{-1}}
		\right),
	\eeq
	and if we set $q_i:=\frac{\Log(\gamma_i)}{2\pi\sqrt{-1}}$, we have
	\beq
	\ol{v}= \left(
		\sum_{i=1}^n b_{i,1}q_i,
		\dots,
		\sum_{i=1}^nb_{i,d}q_i
		\right) = \sum_{i=1}^nq_i\ol{b_i} = \sum_{\rho_i\in\sigma}q_i\ol{b_i},
	\eeq
	with the last equality due to the fact that $q_i=0$ whenever $\rho_i\notin\sigma$. So $v\in\Boxx(\stsig)$.
	
	It remains to show that $g\to v$ is onto. Take any $v\in\Boxx(\stsig)$. So we have 
	\beq
	v=\left(\sum_{\rho_i\in\sigma}q_ib_{i,1} , \dots, \sum_{\rho_i\in\sigma}q_ib_{i,d}, p_1,\dots,p_r\right),
	\eeq
	for some $q_i\in[0,1)\cap\Q$, $1\leq i\leq n$ and for some $p_l\in\{0,\dots,m_l-1\}=\Z/m_l$, $1\leq l\leq r$. We will show there is an element of $G_\sigma$ which is sent to $v$ under the correspondence in (\ref{eq:gvcorr}).
	
	Set $\gamma_i:=e^{2\pi\sqrt{-1}q_i}$ and $s_l:=\left(\prod_{i=1}^n e^{-2\pi\sqrt{-1}q_ib_{i,d+l}}\right)^{1/m_l}e^{2\pi\sqrt{-1}p_l/m_l}$. We claim that $(\gamma_1,\dots,\gamma_n,s_1,\dots,s_r)$ is the element of $G_\sigma$ such that $g\leftrightarrow v$.
	
	Indeed, for $1\leq j\leq d$ we have 
	\beq
	\prod_{\rho_i\in\sigma} {\gamma_i}^{b_{i,j}} 
	= \prod_{\rho_i\in\sigma} \left(e^{2\pi\sqrt{-1}q_i}\right)^{b_{i,j}}  
	=\exp\left(2\pi\sqrt{-1}\sum_{\rho_i\in\sigma} q_ib_{i,j}\right)=1;
	\eeq
	for $1\leq l\leq r$ we have
	\begin{align*}
	{s_l}^{m_l}\prod_{\rho_i\in\sigma}{\gamma_i}^{b_{i,d+l}} 
	&= \left(\left(\prod_{i=1}^n e^{-2\pi\sqrt{-1}q_ib_{i,d+l}}\right)^{1/m_l}e^{2\pi\sqrt{-1}p_l/m_l}\right)^{m_l}\prod_{\rho_i\in\sigma}{\gamma_i}^{b_{i,d+l}}   \\
	&= \left(\prod_{i=1}^n e^{-2\pi\sqrt{-1}q_ib_{i,d+l}}\right)e^{2\pi\sqrt{-1}p_l}\prod_{\rho_i\in\sigma}e^{2\pi\sqrt{-1}q_ib_{i,d+l}}   \\
	&= e^{2\pi\sqrt{-1}p_l} \\
	&= 1	;
	\end{align*}
	and for $\rho_i\notin\sigma$ we clearly have $q_i=0$ and thus $\gamma_i=1$.

\end{proof}
	
\begin{remark}
	Along with this statement, we have that every non-zero box element corresponds with a twisted sector in $I_GZ$.
\end{remark}

\begin{example}\label{ex:p64corr}
	Recall the diagram in Figure~\ref{fig:p64} which illustrates the stacky fan for the toric stack $\sX=\PP(6,4)$.
	By Proposition~\ref{prop:box-corr}, we have the following correspondences under (\ref{eq:gvcorr}):
	\beq
	\begin{array}{lcl} 
		\underline{\Boxx(\stSig)} && \underline{G\subset(\C^\times)^3} \\
		v_0=(0,0)&\leftrightarrow&g_0=(1,1,1)\\
		v_1=(1,0) &\leftrightarrow&g_1=(-1,1,e^{\pi\sqrt{-1}/2})\\
		v_2=(1,1) &\leftrightarrow& g_2=(-1,1,e^{3\pi\sqrt{-1}/2})\\
		v_3=(0,1) &\leftrightarrow& g_3=(1,1,-1)\\
		v_4=(-1,0) &\leftrightarrow& g_4=(1,e^{2\pi\sqrt{-1}/3},1)\\
		v_5=(-1,1) &\leftrightarrow& g_5=(1,e^{2\pi\sqrt{-1}/3},-1)\\
		v_6=(-2,0) &\leftrightarrow& g_6=(1,e^{4\pi\sqrt{-1}/3},1)\\
		v_7=(-2,1) &\leftrightarrow& g_7=(1,e^{4\pi\sqrt{-1}/3},-1)\\
		\end{array}
	\eeq
	Note that $\stSig$ contains two one-dimensional cones and one zero-dimensional cone, and by definition 
	\begin{align*}
	\Boxx(\stSig)=&\Boxx(\boldsymbol\sigma_1) \bigcup \Boxx(\boldsymbol\sigma_2)\bigcup \Boxx(\bv 0)\\
	=&\{v_0,v_1,v_2,v_3\} \bigcup \{v_0,v_3, v_4, v_5, v_6, v_7\} \bigcup\{v_0,v_3\}.
	\end{align*}
\end{example}

\begin{remark} \label{rem.doublebox}
The construction of $\ix(\stSig)$ as a quotient $[Z/G]$ implies that
if $v_1, v_2 \in \Boxx(\stSig)$ correspond to elements $g_1, g_2 \in G$
then $Z^{g_1,g_2} \neq \emptyset$ if and only if $v_1, v_2$ lie in a common cone
of the fan $\Sigma$. In particular we can index the double inertia by pairs $\{(v_1, v_2)|
\text{$v_1$ and $v_2$ lie in a common cone of $\sigma$}\}$. We refer to this subset
of\\ $\Boxx(\stSig) \times \Boxx(\stSig)$ as the {\em double box} and denotes it
by $\Boxx^2(\stSig)$.
\end{remark}

\subsection{The Chow group of the inertia as a module for the stacky Stanley-Reisner ring}

All inertial products are defined on the $\SR_{\stSig}$-module
$A^*(I\cX(\stSig))$. (Recall that $\SR_{\stSig}$ is the integral Chow ring of the toric stack $\cX(\stSig)$.) 
Before describing the multiplication for the
inertial Chow rings we recall the module strucutre.  We already know
that components of the inertia stack are indexed by box elements. If
we express $\sX(\stSig)$ as a quotient stack $[Z/G]$ then if $v \in
\Boxx(\stSig)$ corresponds to $g \in G$ then the component of
$I\sX(\stSig)$ corresponding to $v$ is the quotient stack $[Z^g/G]=
\sX(\stSig/\sigma(\vbar))$ where $\sigma(\vbar)$ is the minimal cone
of $\Sigma$ containing $\vbar$ \cite[Lemma 4.6]{BCS:05}.  
The Chow group of this component is a
cyclic $A^*([Z/G])$-module and we can describe this module structure
using a combination of toric and equivariant methods.

\begin{proposition} \label{prop.sectorChow}If $v$ is a box element then there is an isomorphism of $\SR_\stSig$-modules 
$A^*(\cX(\stSig/\sigma(\vbar))) \simeq \SR_\stSig/I_{\stSig,v}$
where $I_{\stSig,v}$ is the ideal of $\SR(\stSig)$ defined by
the relations
$$\{\widetilde{x}_{i_1}, \ldots , \widetilde{x}_{i_k}: \rho_{i_1} , \ldots ,\rho_{i_k} \text{do not lie in a cone of }\Starr(\sigma(\vbar)\}.$$
where $\Starr{\sigma(\vbar)}$ refers to the cones of $\Sigma$ containing $\sigma(\vbar)$. 
\end{proposition}
\begin{proof}
The stack $\cX(\stSig) = [Z/G]$ where $Z = \A^d \smallsetminus V(J_{\Sigma})$
where $J_{\Sigma}:=\vr{\left.\prod_{\rho_i\not\subset\sigma}z_i\right|
  \sigma\in\Sigma}$ is the irrelevant ideal. Now if $g \in G$ corresponds to
a box element $v$ with minimal cone $\sigma(\vbar)$ then 
$Z^g = V(z_{i_1, \ldots , i_k})$ where $\rho_{i_1}, \ldots \rho_{i_k}$ are the rays
of $\sigma(\vbar)$. The proposition now follows from 
Proposition \ref{prop.relativecase.xxx}.
\end{proof}

\section{Toric computations} 


Let $\stSig=(N, \Sigma,\beta)$ be a stacky fan with associated toric
stack $\sX=\cX(\stSig)=[Z/G]$, as constructed in
Section~\ref{ssec:fans,stacks}. 
Let
$g=(\gamma_1,\dots,\gamma_n,s_1,\dots,s_r)\in G$ act on $Z$ by the
diagonal action of the first $n$ components, as defined in
Section~\ref{ssec:fans,stacks}.

	Let $\mathcal{L}_1,\dots,\mathcal{L}_n$ be the line bundles 
associated to the divisors where $z_i=0$
as in Section \ref{sec.charline}
. Writing $\ix(\stSig) = [Z/G]$ as above then
$\T Z = \mathcal{L}_1 + \ldots + \mathcal{L}_n$ and the vector bundle
$\T\ix$ corresponding to the tangent bundle of $\ix$ fits into an 
exact sequence of equivariant vector bundles \cite[Lemma A.1]{EdGr:05}
$$0 \to T\ix \to TZ \to \Lie(G)$$
Since $G$ is diagonalizable $\Lie(G)$ is the trivial bundle so
$T\ix$ can be written in $K$-theory as $\sum_{i=1}^n\LL_i - m\one$
where $\one$ is the class of the trivial one-dimensional representation of $G$
and $m$ is a non-negative integer.
A similar analysis shows that the class corresponding to $I\ix(\stSig)$ can be written on each component of $I\ix(\stSig)$ as a sum of a subset of the $\mathcal{L}_i$ minus  a positive multiple of $\one$. Since logarithmic trace of the trivial bundle $\one$ is 0, we can make computations as if $\T\ix = \T Z = \sum_{i=1}^n \LL_i$.


The following Lemma tells us which of the $\mathcal{L}_i$ appear in the expression for $\T I\ix(\stSig)$ on each component.	
\begin{lemma}\label{lem:equiv}
	Let $\sX$ be the stack $\cX(\stSig)=[Z/G]$ associated to the stacky fan $\stSig=(N, \Sigma, \beta)$. Let $I\sX=[I_GZ/G]$ be the inertia stack, and $\T_{I\sX}$ be the class in $K_G(I_GZ)$ corresponding to the tangent bundle. 
Then for any $i$, $1\leq i\leq$n, and for any box element $v\in\Boxx(\stSig)$ and its corresponding $g=(\gamma_1,\dots,\gamma_n,s_1,\dots,s_r)\in G$, the following statements are equivalent:
		\benu
		\item $q_i\neq0$
		\item $\rho_i\subseteq\sigma(\ol v)$
		\item $\gamma_i\neq1$
		\item $z_i=0$ for all $z=(z_1,...,z_n)\in Z^g$
		\item $\mathcal{L}_i$ has coefficient 0 in $\left.\left(\T_{I\sX}\right)\right|_{Z^g}$
		\eenu
\end{lemma}

\begin{proof}
  Lemma 4.6 of \cite{BCS:05} shows that the first four statements are
  equivalent when $\Sigma$ is a complete simplicial fan, so it will
  suffice for us to show that the fifth is equivalent to the
  fourth. The normal bundle to the map $I_GZ \to Z$ is $\sum{\mathcal{L}_i}$, where
  the sum is over all $i$ such that $z_i=0$ for all $z\in Z^g$. (Note that
the map $I_GX \to X$ is a closed embedding on each component of $I_GZ$.)
Thus,
  the restricion of the tangent bundle is
  $\left.\left(\T_{I_GZ}\right)\right|_{Z^g}=\sum{\mathcal{L}_i}$,
  where the sum is over all $i$ such $z_i\neq0$ for at least one $z\in
  Z^g$. That is, the coefficient on $\mathcal{L}_i$ is zero if and
  only if $z_i=0$ for all $z\in Z^g$.
\end{proof}

\begin{example}
  Consider the stack $\sX=\PP(6,4)$ of Example~\ref{ex:p64corr}. We
  have $\ol{v_5}=-1=0\ol{b_1}+\frac13\ol{b_2}$, so $q_{5,1}=0$ and
  $q_{5,2}\neq0$ (where the notation $q_{j,i}$ describes the rational
  coefficient for $\ol{v_j}$ on $b_i$). That is to say,
  $\sigma(\ol{v_5})$ contains $\rho_2$ and not $\rho_1$. The
  corresponding group element $g_5=(1,e^{2\pi\sqrt{-1}/3},-1)\in G$
  indeed has $\gamma_1=1$ and $\gamma_2\neq1$.  Thus
  $Z^{g_5}=\{(z_1,0)\st z_1\neq0\}$. Finally, the restriction of the tangent bundle here
  is just $\mathcal{L}_1$. In short, the five conditions of
  Lemma~\ref{lem:equiv} are true for $i=2$ and false for $i=1$.
	
  On the other hand, if we instead take the box element $v_2=(1,1)$
  corresponding to $g_2=(-1,1,e^{3\pi\sqrt{-1}/2})$, the five
  conditions of Lemma~\ref{lem:equiv} are true for $i=1$ and false for
  $i=2$; accordingly, the restriction of the tangent bundle is $\mathcal{L}_2$.
	
	Lastly, if we take the box element $v_3=(0,1)$ corresponding to $g_3=(1,1,-1)$, we have all five conditions false for both $i=1$ and $i=2$, and the 
restriction of the tangent bundle is $\mathcal{L}_1+\mathcal{L}_2$. This last example illustrates an important type of box element, one in which $q_{j,i}=0$ for all $i$. These correspond to elements $g \in G$ which fix all of $Z$.
      \end{example}

\begin{corollary}\label{cor:ev-pullback}
	If $e_1:\T I^2\ix \to \T I\ix$ and $e_2:\T I^2\ix\to\T\ix$ are the evaluation maps and $(v_1,v_2)\in\Boxx(\stSig)\times\Boxx(\stSig)$, then the coefficient of $\mathcal{L}_i$  in $\left.\left(e_1^*\T_{I_GZ}\right)\right|_{Z^{g_1,g_2}}$ (resp. $\left.\left(e_2^*\T_{I_GZ}\right)\right|_{x^{g_1,g_2}}$) is 1 if and only if $q_{1,i}=0$ (resp. $q_{2,i}=0$).
\end{corollary}

As a corollary we can also descrbe which $\mathcal{L}_i$ appear in the expression for $\T I^2\ix$
\begin{corollary}\label{cor:TI2GX}
	The coefficient of $\mathcal{L}_i$ in $\T_{I^2_GZ}$ associated to $(v_1,v_2)\in\Boxx(\stSig)\times\Boxx(\stSig)$ is 1 if and only if $q_{1,i}=q_{2,i}=0$.
\end{corollary}

\begin{proof}
	Adapting Lemma~\ref{lem:equiv}, we have that a point $(z,g_1,g_2)\in\mathbb{I}_G^2X$ is in the fixed locus of both $g_1=(\gamma_{1,1},\dots,\gamma_{1,n},s_{1,1},\dots,s_{1,r})$ and $g_2=(\gamma_{2,1},\dots,\gamma_{2,n},s_{2,1},\dots,s_{2,r})$ if and only if
	\beq
	z\in\{(z_1,...,z_n)\in Z\st z_i=0 \text{ if } \gamma_{1,i}\neq1\text{ or }\gamma_{2,i}\neq1\}
	\eeq
	\beq
	=\{(z_1,...,z_n)\in Z\st z_i=0\text{ if }q_{1,i}\neq0\text{ or }q_{2,i}\neq0\}
	\eeq
	The first presentation of this set indicates that the tangent bundle to this set at $(z,g_1,g_2)$ is 
	\beq
	\T_{I_G^2Z}=\dsp\sum_{\gamma_{1,i}=\gamma_{2,i}=1}\mathcal{L}_i\ .
	\eeq
	The second presentation indicates that the indexing set on this sum is equivalent to the set of $i$ where $q_{1,i}=q_{2,i}=0$.
\end{proof}

\begin{example}
	Take our usual example, $\ix=\PP(6,4)$, in the notation as in Example~\ref{ex:p64corr}. Then for the pair $(v_4,v_3)\in\Boxx(\stSig)\times\Boxx(\stSig)$, we have $e_1^*\T_{I_GZ}=\mathcal{L}_2$ and $e_2^*\T_{I_GZ}=\mathcal{L}_1+\mathcal{L}_2$ by Corollary~\ref{cor:ev-pullback}, and $\T_{I_G^2Z}=\mathcal{L}_2$ by Corollary~\ref{cor:TI2GX}.
\end{example}

\begin{lemma}\label{lem:box-trace}
	Let $v$ be the box element corresponding to some $g\in G$, where $\ol{v}=\sum_{i=1}^nq_i\ol{b_i}\in\Boxx(\stSig)$ for some $0\leq q_i<1$, $i=1, 2,\dots, n$. Then for all $i\in\{1,\dots,n\}$,
	\beq
	L(g)(\mathcal{L}_i)=q_i\mathcal{L}_i,
	\eeq
	where $L(g)$ is the logarithmic trace of Definition~\ref{def:logtrace}.
\end{lemma}

\begin{proof}
	For the usual action of $g=(\gamma_1,\dots,\gamma_n,s_1,\dots,s_r)$
	\beq
	g\cdot(z_1,\dots,z_n)=(\gamma_1z_1,\dots,\gamma_nz_n),
	\eeq
	the eigenvalue on the bundle $\mathcal{L}_i$ is $\gamma_i$. By Proposition~\ref{prop:box-corr}, we have $\gamma_i=e^{2\pi\sqrt{-1}q_i}$. Thus $q_i\mathcal{L}_i$ is the logarithmic trace.
\end{proof}

\begin{example}
	Once again recycling Example~\ref{ex:p64corr}, we have $L(g_2)(\mathcal{L}_1)=\frac12\mathcal{L}_1$ and $L(g_2)(\mathcal{L}_2)=0$. Similarly, $L(g_7)(\mathcal{L}_1)=0$ and $L(g_7)(\mathcal{L}_2)=\frac23\mathcal{L}_2$. We also have $L(g_3)(\mathcal{L}_1)=L(g_3)(\mathcal{L}_2)=0$, which nicely illustrates the fact that $L(g)$ depends on $\ol{v}$ but not on $v$. That is to say, logarithmic trace ignores torsion.
\end{example}

We now compute the classes ${\mathcal{L}_i}^+$ and ${\mathcal{L}_i}^-$ in $K_G(\I_G^2Z)$ of Proposition~\ref{def:E-plus-minus}.

\begin{lemma}\label{lem:computeLi+}
	Let $\mathcal{L}_i$ be the $G$-equivariant line bundle associated to the $i$-th component of the action of $G$ on $Z$. For $j=1,2$, let $g_j=(\gamma_{j,1},\dots,\gamma_{j,n},s_{j,1},\dots,s_{j,r})$ be an element of $G$, and let $v_j$ be its corresponding element in $\Boxx(\stSig)$, where $q_{j,i}$ is defined by $\ol{v_j}=\sum_{i=1}^nq_{j,i}\ol{b_{j,i}}$ (for $1\leq i\leq n$ and $j=1, 2$).
	
	Then for each $i$, $1\leq i\leq n$, exactly one of the three cases holds: 
	\benu
	\item 
		\benu
		\item At least one of $q_{1,i},q_{2,i}$ is zero, and
		\item ${\mathcal{L}_i}^+(g_1,g_2)=0$, and
		\item ${\mathcal{L}_i}^-(g_1,g_2)=0$.
		\eenu
	\item 
		\benu
		\item $q_{1,i}+q_{2,i}<1$ with $q_{1,i},q_{2,i}$ both nonzero, and
		\item ${\mathcal{L}_i}^+(g_1,g_2)=0$, and
		\item ${\mathcal{L}_i}^-(g_1,g_2)=\mathcal{L}_i$.
		\eenu
	\item 
		\benu
		\item $q_{1,i}+q_{2,i}\geq1$, and 
		\item ${\mathcal{L}_i}^+(g_1,g_2)=\mathcal{L}_i$, and
		\item ${\mathcal{L}_i}^-(g_1,g_2)=0$.
		\eenu
	\eenu
\end{lemma}

\begin{proof}
	Since $0\leq q_{i,j}<1$ for any $i$ and $j$, the cases 1(a), 2(a) and 3(a) are disjoint and encompass all possibilities. So we must only show that in each of the three cases, (a) implies both (b) and (c).
	
	For case 1, without loss of generality we may assume $q_{1,i}=0$. Then by Lemma~\ref{lem:box-trace}, $\mathcal{L}_i(g_1)=\mathcal{L}_i(g_1^{-1})=0$, and thus we have 
	\beq
	{\mathcal{L}_i}^+(g_1,g_2)=L(g_2)(\mathcal{L}_i|_{Z^{g_1,g_2}})-L(g_1g_2)(\mathcal{L}_i|_{Z^{g_1,g_2}})
	\eeq
	and
	\beq
	{\mathcal{L}_i}^-(g_1,g_2)=L(g_2^{-1})(\mathcal{L}_i|_{Z^{g_1,g_2}})-L(g_2^{-1}g_1^{-1})(\mathcal{L}_i|_{Z^{g_1,g_2}}).
	\eeq
	Again by Lemma~\ref{lem:box-trace}, $L(g_1g_2)(\mathcal{L}_i)=q_{2,i}\mathcal{L}_i=L(g_2)(\mathcal{L}_i)$ and $L(g_2^{-1}g_1^{-1})=(1-q_{2,i})\mathcal{L}_i=L(g_2^{-1})(\mathcal{L}_i)$, which completes the proof of case 1.

	For cases 2 and 3, we have
	\beq
	L(g_1g_2)(\mathcal{L}_i)=	
	\left\{	
		\begin{array}{lll}
		(q_{1,i}+q_{2,i})\mathcal{L}_i & \text{if }&0< q_{1,i}+q_{2,i}<1\\
		(q_{1,i}+q_{2,i}-1)\mathcal{L}_i& \text{if }&q_{1,i}+q_{2,i}\geq1\\
		\end{array}	
	\right.
	\eeq
	and 
	\beq
	L(g_2^{-1}g_1^{-1})(\mathcal{L}_i)=	
		\left\{	
			\begin{array}{lll}
			(1-q_{1,i}-q_{2,i})\mathcal{L}_i & \text{if }&0<q_{1,i}+q_{2,i}\leq1\\
			(2-q_{1,i}-q_{2,i})\mathcal{L}_i & \text{if }&q_{1,i}+q_{2,i}>1
			\end{array}	
		\right..
	\eeq
	Thus,
	\beq
	\begin{array}{rcl}
		{\mathcal{L}_i}^+(g_1,g_2)
		&=&	
		L(g_1)(\mathcal{L}_i|_{Z^{g_1,g_2}})+L(g_2)(\mathcal{L}_i|_{Z^{g_1,g_2}})-L(g_1g_2)(\mathcal{L}_i|_{Z^{g_1,g_2}})\\[2mm]
		&=&
		\left\{	
			\begin{array}{lll}
			0 & \text{if }&0< q_{1,i}+q_{2,i}<1\\
			\mathcal{L}_i& \text{if }&q_{1,i}+q_{2,i}\geq1\\
			\end{array}	
		\right.\end{array}
	\eeq
	and 
	\beq
	\begin{array}{rcl}
		{\mathcal{L}_i}^-(g_1,g_2)
		&=&
		L(g_1^{-1})(\mathcal{L}_i|_{Z^{g_1,g_2}})+L(g_2^{-1})(\mathcal{L}_i|_{Z^{g_1,g_2}})-L(g_2^{-1}g_1^{-1})(\mathcal{L}_i|_{Z^{g_1,g_2}})\\[2mm]
		&=&
		\left\{	
			\begin{array}{lll}
			\mathcal{L}_i & \text{if }&0< q_{1,i}+q_{2,i}<1\\
			0& \text{if }&q_{1,i}+q_{2,i}\geq1\\
			\end{array}	
		\right..\end{array}
	\eeq
\end{proof}

\begin{definition}\label{def:B+-}
	For any pair $v_1,v_2\in\Boxx(\stSig)$, define the indexing sets $B^+_\stSig(v_1,v_2)$ and $B^-_\stSig(v_1,v_2)$ to be the following subsets of $\{1,2,\dots,n\}$:
	\beq
	B_\stSig^+(v_1,v_2):=\{i\st q_{1,i}+q_{2,i}\geq1\}
	\eeq
	\beq
	B_\stSig^-(v_1,v_2):=\{i\st q_{1,i}+q_{2,i}<1\text{ and } q_{1,i},q_{2,i}\neq0\}.
	\eeq
\end{definition}

\begin{proposition}\label{prop:computeV+}
	Let $V$ be the $G$-equivariant $Z$-bundle $V=\sum_{i=1}^n a_i\mathcal{L}_i$, where $a_i$ is a non-negative integer for each $i$. Then for any pair $v_1,v_2\in\Boxx(\stSig)$,
	\beq
	V^+(g_1,g_2)
	=\dsp\sum_{i=1}^n a_i {\mathcal{L}_i}^+(g_1,g_2)
	=\dsp\sum_{i \in B^+_{\stSig}(v_1,v_2)}a_i\mathcal{L}_i
	\eeq
	and
	\beq
	V^-(g_1,g_2)
	=\dsp\sum_{i=1}^n a_i {\mathcal{L}_i}^-(g_1,g_2)
	=\dsp\sum_{i \in B^-_\stSig(v_1,v_2)}a_i\mathcal{L}_i.
	\eeq
\end{proposition}

\begin{proof}
	Since the logarithmic trace of a vector bundle is additive, if $V=\sum_{i=1}^na_i\mathcal{L}_i$ then for any $g\in G$,  
	\beq
	L(g)(V)=L(g)\left(\sum_{i=1}^na_i\mathcal{L}_i\right)=\sum_{i=1}^nL(g)(a_i\mathcal{L}_i)=\sum_{i=1}^na_iL(g)(\mathcal{L}_i).
	\eeq
	The statement then follows directly from Lemma~\ref{lem:computeLi+}.
\end{proof}

\begin{figure}\label{fig:p654box}
 	$$\includegraphics[width=3in]{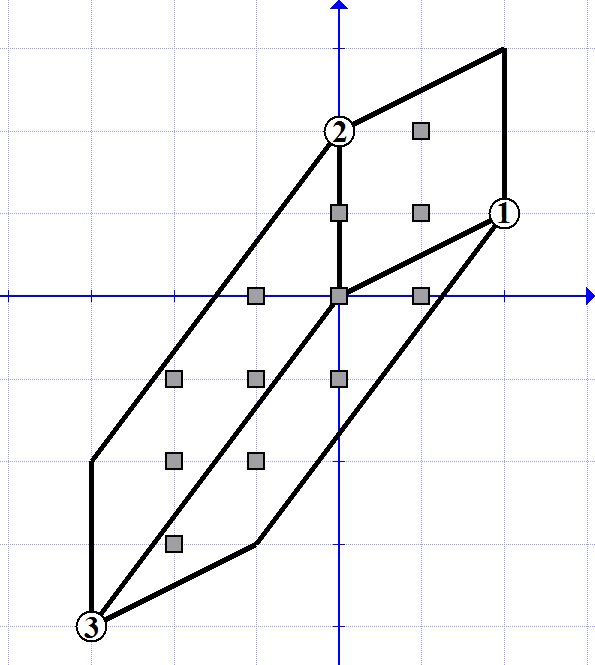}$$
 	\caption{The box diagram for the stacky fan in Example~\ref{ex:p654} which produces $\PP(6,5,4)$. }
 \end{figure}

\begin{example}\label{ex:p654}
	The definitions of $B^+_\stSig$ and $B^-_\stSig$ may not seem particularly intuitive until we look at an example which includes a box diagram. 
	
	Consider the stacky fan $\stSig$ with $N=\Z^2$, with the rays of $\Sigma$ generated by $b_1=(2,1)$, $b_2=(0,2)$ and $b_3=(-3,-4)$. Figure~\ref{fig:p654box} displays $\Boxx(\stSig)$, in a way such that the elements in $\Boxx(\stSig)$ (denoted by gray squares) are ``boxed off'' by the parallelograms formed by the $b_i$. Any element of $N=\Z^2$ outside the parallelograms is equivalent (modulo some $N(\stsig)$) to a box element lying inside one of the three parallelograms. 
	
	Consider the cone $\sigma_1$ formed by $\rho_1$ and $\rho_2$. We'll name the three non-zero box elements $v_1=(0,1)$, $v_2=(1,1)$ and $v_3=(1,2)$. Then $\ol{v_1}=\frac12\ol{b_2}$, $\ol{v_2}=\frac12\ol{b_1}+\frac14\ol{b_2}$ and $\ol{v_3}=\frac12\ol{b_1}+\frac34\ol{b_2}$.
	
	The intuitive way to think about  $B^+_\stSig$ is that it lists the "directions" in which the sum of the box elements "leaves" these parallelograms. For instance, with the pair $v_1$ and $v_2$, $B^+_\stSig(v_1,v_2)=\emptyset$, because $(0,1)+(1,1)=(1,2)$ which is still inside the parallelogram. On the other hand, $B^+_\stSig(v_2,v_3)=\{1,2\}$, since $(1,1)+(1,2)=(2,3)$, which leaves the parallelogram in the direction of both $b_1$ and $b_2$. Similarly, $B^+_\stSig(v_2,v_2)=\{1\}$ because the sum $v_2+v_2$ leaves the box only in the direction of $b_1$.
	
	The other side of the coin is $B^-_\stSig$, which lists the ways in which the sum does \emph{not} leave the parallelogram. For example, we have $B^-_\stSig(v_2,v_3)=\emptyset$ and $B^-_\stSig(v_2,v_2)=\{2\}$. An oddity occurs when one of the box elements happens to be on a ray, such as $B^-_\stSig(v_1,v_2)=\{2\}$; the index 1 is not included because $v_1$ is on the ray $b_2$ (and therefore $q_{1,1}=0$).
	
	Let $V=a_1\mathcal{L}_1+a_2\mathcal{L}_2+a_3\mathcal{L}_3$ for some non-negative integers $a_1, a_2, a_3$. Then we can compute the following, using Proposition~\ref{prop:computeV+} in conjunction with the above calculations:
	\beq
	\begin{array}{ll}
	V^+(v_1,v_2) = 0 & V^-(v_1,v_2) = a_2\mathcal{L}_2\\[1mm]
	V^+(v_2,v_2) = a_1\mathcal{L}_1 & V^-(v_2,v_2) = a_2\mathcal{L}_2\\[1mm]
	V^+(v_2, v_3) = a_1\mathcal{L}_1+a_2\mathcal{L}_2 \hspace{.2in} & V^-(v_2,v_3) = 0\\[1mm]
	\end{array}
	\eeq
	
\end{example}

\section{Presentations for inertial Chow rings}\label{ssec:chow-results}

	
	Let $\cX(\stSig)$ be a toric stack with stacky fan $\stSig=(N,\Sigma,\beta)$, and let $b_1,\dots,b_n$ be the distinguished points in $N$ defined by $\beta$. Let $\Boxx(\stSig)=\{v_1,\dots,v_k\}$. 

Let $\SR_{\stSig}$ be the stacky Stanley-Reisner ring of the stacky fan, which equals the intergral Chow ring of the stack $\cX(\stSig)$. 
The following is an immediate consequence of Proposition \ref{prop.sectorChow}
\begin{proposition}\label{prop:underlying-module-imm}
	Let $\cX(\stSig)$ be a toric Deligne-Mumford stack associated to the stacky fan $\stSig$, then for any inertial product we have the following isomorphism of $\SR_{\stSig}$-modules:
		\beq
		A^*(\cX(\stSig),\star,\Z)
		\cong 
		\bigoplus_{v\in\Boxx(\stSig)}\SR_{\stSig}/I_{\stSig,v}
		\eeq
where $I_{\stSig,v}$ is the ideal defined in the statement of Proposition \ref{prop.sectorChow}.
\end{proposition}
	
We now work to define the algebra structure on $A^*(\ix(\stSig), \star, \Z)$.
We begin by defining the ring $R_\stSig$ as the quoteint $\SR_{\stSig}[\{y^v\}_{v \in  \Boxx(\stSig)}]/
\sum_v I_{\stSig,v}y^v$. Every inertial Chow ring will be a quotient of $R_\stSig$.

\begin{proposition} \label{prop.conerelations}
Let $v_1$ and $v_2$ be elements of $\Boxx(\stSig)$. 
If $\star$ is any inertial product then $y^{v_1} \star y^{v_2} = 0$
if $v_1, v_2$ do not lie in a common cone.
\end{proposition}
\begin{proof}
Writing $\ix(\stSig) = [Z/G]$ then  $v_1,v_2$ correspond to elements
 $g_1, g_2 \in G$ such that $Z^{g_1}, Z^{g_2} \neq \emptyset$. However $Z^{g_1,g_2}= \emptyset$ if $v_1, v_2$ do not lie in a common cone. Thus $e_1^*y^{v_1} =e_2^*y^{v_2} = 0$ so $y^{v_1} \star y^{v_2} = 0$
\end{proof}
Thanks to Proposition \ref{prop.conerelations} we make the following definition.

	\begin{definition} \label{def.conerelations}
	We define the \emph{cone relations ideal} in $R_\stSig$ to be
	\begin{equation*}
	\CR(\stSig):= \vr{ y^{v_i}y^{v_j} \st v_i, v_j\in\Boxx(\stSig) \text{ and no cone contains both }v_i\text{ and }v_j  }.
	\end{equation*}
	\end{definition}

\begin{proposition} \label{prop.twistfunction}
  Let $\cX(\stSig)$ be a toric stack and let $v_1, v_2 \in
  \Boxx(\stSig)$ such that $v_1, v_2$ lie in a common cone
  $\sigma$. Let $v_3$ be the box element representing $v_1 + v_2$.
and corresponding to $g_3 \in G$.
If $\star$ is any inertial product then there exists a class
$\Tw(\star)(v_1,v_2) \in A^*_G(Z^{g_3})$ such that 
\begin{equation}\label{eq:twfn}
				y^{v_1}\star y^{v_2} = y^{v_3}\cdot
					\Tw(\star)(v_1,v_2)\cdot
				\prod_{q_{1,i}+q_{2,i}=1}c_1(\mathcal{L}_i),
				\end{equation}
\end{proposition} 
\begin{proof}
With notation that we have established, the class $\prod_{q_{1,i}+q_{2,i}=1}c_1(\mathcal{L}_i)$ is the class
of $\mu_*([Z^{g_1,g_2}])$ in $A^*(Z^{g_3})$. Moreover, $A^*(Z^{g_1,g_2})$ is generated as an algebra by $A^*(Z^{g_3})$. Hence by the projection formula
$$\mu_*\left(e_1^* y^{v_1} \cdot e_2^*(y^{v_2}) \cdot c)\right)
= y^{v_3}\cdot
					\Tw(\star)(v_1,v_2)\cdot \prod_{q_{1,i}+q_{2,i}=1}c_1(\mathcal{L}_i),$$
for some class $\Tw(\star)(v_1,v_2) \in A^*(Z^{g_3})$ whose restriction to
$A^*(Z^{g_1,g_2})$ is $c$.				
\end{proof}

\begin{definition}
  Let $\star$ be an inertial product on the toric stack
  $\sX=\cX(\stSig)$ associated to the stacky fan $\stSig=(N,
  \Sigma,\beta)$.  The \emph{twisting
    function for the inertial product $\star$} is the mapping \beq
  \Tw(\star):\Boxx^2(\stSig) \to A^*(\I\sX) \eeq
  which sends a box pair $(v_1, v_2)$ to the class completely
  determined by $\star$ in $A^*(\I\sX)$. (Recall  - Remark \ref{rem.doublebox}
- that $\Boxx^2(\stSig)$ is the set of pairs $(v_1,v_2)$ such that $v_1, v_2$
lie in a common cone.) Its component at 
$g_3 = g_1g_2$, where
  $g_1$ and $g_2$ are the elements of $G$ corresponding to $v_1$ and
  $v_2$ is given as follows:
				\begin{equation}\label{eq:twfn}
				y^{v_1}\star y^{v_2} = y^{v_3}\cdot
					\Tw(\star)(v_1,v_2)\cdot
				\prod_{q_{1,i}+q_{2,i}=1}c_1(\mathcal{L}_i),
				\end{equation}
		where $v_3$ is the unique box element such that $v_1+v_2=v_3$ under box addition. 
\end{definition}

As a first
example, we'll compute the twisting function for the orbifold product
$\star_{orb}$ of Definition~\ref{def:orb-prod}.

\begin{example}
  By Definition~\ref{def:orb-prod}, the orbifold product $\star_{orb}$ is given by $\star_c$ where $c=\eu(LR(\T))=\eu(LR(\T))$. Then for any pair
  $(v_1,v_2)\in\Boxx^2(\stSig)$ with corresponding pair $g_1,g_2\in G$, let $v_3=v_1+v_2$ under box addition. Then we have 
    \beq
    y^{v_1}\star_{orb}y^{v_2}=y^{v_3}\cdot
    \Tw(\star_c)(v_1,v_2)\cdot
	\prod_{q_{1,i}+q_{2,i}=1}c_1(\mathcal{L}_i)
    \eeq
    \beq
    = y^{v_3}
    \cdot \eu(LR(\T)(g_1,g_2,(g_1g_2)^{-1}))
	\prod_{q_{1,i}+q_{2,i}=1}c_1(\mathcal{L}_i)
    \eeq    
    \beq
    =y^{v_3}\cdot
    \prod_{q_{1,i}+q_{2,i}>1}c_1(\mathcal{L}_i)
	\prod_{q_{1,i}+q_{2,i}=1}c_1(\mathcal{L}_i)
    \eeq  
    \beq
    =y^{v_3}\cdot
    \prod_{q_{1,i}+q_{2,i}\geq1}c_1(\mathcal{L}_i)
    \eeq  
	Thus, by (\ref{eq:twfn}) the twisting function for the orbifold product is: 
		\beq
		\Tw(\star_{orb})(v_1,v_2)=\prod_{q_{1,i}+q_{2,i}>1}c_1(\mathcal{L}_i) = \prod_{q_{1,i}+q_{2,i}>1}\widetilde{x}_i.
		\eeq
\end{example}

The purpose of the twisting function, illustrated in this example, is
to give a succinct way of describing exactly how $c$ affects the
inertial product. We'll use this in the next definition as we continue
our quest to provide a ring presentation for the inertial Chow ring.

\begin{definition}
	Define $\BR(\star,\stSig)$ to be the ideal 
	\beq
	\vr{\left.y^{v_1}y^{v_2}- y^{v_3}\cdot \Tw(\star)(v_1,v_2)\cdot\prod_{q_{1,i}+q_{2,i}=1} \widetilde{x}_i\right| v_1,v_2\in\Boxx^2(\stSig) }.
	\eeq
\end{definition}

\begin{theorem}\label{thm:inertial-chow-ring}
	For any toric stack $\cX(\stSig)$ and any inertial product $\star$ of the form $\star_{V^+}$ or $\star_{V^-}$ for some $G$-equivariant vector bundle $V$, we have an isomorphism of $\Z$-graded rings
	\beq
	A^*(\cX(\stSig),\star,\Z)\cong\dfr{R_\stSig}{\CR(\stSig)+\BR(\star,\stSig)},
	\eeq
	where the isomorphism is given by $y^{b_i}\mapsto c_1(\mathcal{L}_i)$.
\end{theorem}

\begin{proof}
This follows from Proposition \ref{prop.conerelations} and 
Proposition \ref{prop.twistfunction}.
\end{proof}

We now compute the BR ideal for some of the inertial products defined
in \cite{EJK:16}.
	
\begin{proposition}\label{prop:br-v-plus-minus}
Let $\stSig$ be a stacky fan. For any $G$-equivariant bundle of the form 
$V=\sum a_i\LL_i$ with $a_i \geq 0$,  
the box relations ideal for the $\star_{V^+}$ product is
	\begin{eqnarray*}
          \BR(\star_{V^+},\stSig) & =& 
          \vr{\left.y^{v_1}y^{v_2}- y^{v_3}\cdot\prod_{
                i \in B^+_\stSig(v_1,v_2)} \eu(\mathcal{L}_i+a_i\mathcal{L}_i)\right| (v_1,v_2)\in\Boxx^2(\stSig) }\\
& = & \vr{\left.y^{v_1}y^{v_2}- y^{v_3}\cdot\prod_{
                i \in B^+_\stSig(v_1,v_2)} \tilde{x_i}^{a_i + 1}\right| (v_1,v_2)\in\Boxx^2(\stSig) },
          \end{eqnarray*}
          and the box relations ideal for the $\star_{V^-}$ product is
          \begin{eqnarray*}
          \BR(\star_{V^-},\stSig) & =& \vr{\left.y^{v_1}y^{v_2}- y^{v_3}\cdot\prod_{i \in B^-_\stSig(v_1,v_2)} \eu(a_i(\mathcal{L}_i)\cdot\prod_{j \in B^+_\stSig(v_1,v_2)}c_1(\mathcal{L}_j) \right| (v_1,v_2)\in\Boxx^2(\stSig) }\\
& =  & \vr{\left.y^{v_1}y^{v_2}- y^{v_3}\cdot\prod_{i \in B^-_\stSig(v_1,v_2)} \tilde{x}_i^{a_i} \cdot\prod_{j \in B^+_\stSig(v_1,v_2)}\tilde{x}_j \right| (v_1,v_2)\in\Boxx^2(\stSig) }
.
          \end{eqnarray*}
\end{proposition}

\begin{proof}
	For any $v_1,v_2\in\Boxx(\stsig)$ for some cone $\sigma\in\Sigma$, we have a corresponding $g_1,g_2\in G$ and we can write $\ol{v_1}=\sum_{\rho_i\in\sigma}q_{1,i}\ol{b_i}$ and $\ol{v_2}=\sum_{\rho_i\in\sigma}q_{2,i}\ol{b_i}$. Let $v_3=v_1+v_2$ under addition in $\Boxx(\stsig)$. 
	
	We first compute the twisting function for the $\star_{V^+}$ product.  To do so, we need the Euler class of $\mathscr{R}^+V$, using Proposition~\ref{def:E-plus-minus} and Proposition~\ref{prop:computeV+}. We have
	\begin{align*}
	\left(R^+V+LR(\T)\right)(g_1,g_2)
	= V^+(g_1,g_2)+LR(\T)(g_1,g_2) 
	= \sum_{i \in B^+_\stSig(v_1,v_2)} a_i\mathcal{L}_i + \sum_{q_{1,j}+q_{2,j}>1}\mathcal{L}_j,
	\end{align*}
	and so
	\beq
	\Tw(\star_{V^+})(v_1,v_2)=\eu\left(\sum_{i \in B^+_\stSig(v_1,v_2)} a_i\mathcal{L}_i + \sum_{q_{1,j}+q_{2,j}>1}\mathcal{L}_j\right) 
	\eeq
	\beq
	=\prod_{B^+_\stSig(v_1,v_2)} \eu(a_i\mathcal{L}_i))\cdot \prod_{q_{1,i}+q_{2,i}>1}c_1(\mathcal{L}_i).
	\eeq

	The normal bundle for $\bv g=(g_1,g_2)$ is $\sum \mathcal{L}_i$, where the sum is over all $i$ such that $\rho_i\in(\ol{v_1},\ol{v_2})$ but $\rho_i\notin\sigma(\ol{v_3})$. Equivalently, we could say the sum is over all $i$ such that $q_{1,i}+q_{2,i}=1$. 
	
Then we have
	\begin{align*}
	y^{v_1}\star_{V^+}y^{v_2} 
	= & \mu_*(e_1^*y^{v_1}\cdot e_2^*y^{v_2}\cdot \eu(\mathscr{R}^+V))\\
	= & y^{v_3}\cdot\left(\prod_{i \in B^+_\stSig(v_1,v_2)} 
\eu(a_i\mathcal{L}_i)\cdot \prod_{q_{1,j}+q_{2,j}>1}c_1(\mathcal{L}_j)\right)\cdot \prod_{q_{1,k}+q_{2,k}=1} c_1(\mathcal{L}_k)\\
	= & y^{v_3}\cdot\prod_{i \in B^+_\stSig(v_1,v_2)} \eu(\mathcal{L}_i+a_i\mathcal{L}_i)
	\end{align*}
where the last equality follows because $B^+_{\stSig}(v_1, v_2) =\{i | q_{i,1} +q_{i,2} \geq 1\}$.
	
	Thus,
	\beq
	\BR(\star_{V^+},\stSig)=\vr{\left.y^{v_1} y^{v_2}- y^{v_3}\cdot\prod_{i \in B^+_\stSig(v_1,v_2)} \eu(\mathcal{L}_i+a_i\mathcal{L}_i) \right| v_1,v_2\in\Boxx^2(\stSig) }.
	\eeq

	If we instead consider the $\star_{V^-}$ product, we instead have the twisting function
	\beq
	\Tw(\star_{V^-})(v_1,v_2)=\eu\left(\sum_{i \in B^-_\stSig(v_1,v_2)} a_i\mathcal{L}_i + \sum_{q_{1,j}+q_{2,j}>1}\mathcal{L}_j\right)
	=\prod_{i \in B^-_\stSig(v_1,v_2)} \eu(a_i\mathcal{L}_i)\cdot
\prod_{q_{1,j}+q_{2,j}>1}c_1(\mathcal{L}_j)
	\eeq
	and so 
	\begin{align*}
	w_1\star_{V^-}w_2 
	= & \mu_*(e_1^*w_1\cdot e_2^*w_2\cdot \eu(\mathscr{R}^-V))\\
	= & w_3\cdot\left(\prod_{i \in B^-_\stSig(v_1,v_2)} \eu(a_i\mathcal{L}_i)
\cdot\prod_{q_{1,j}+q_{2,j}>1}c_1(\mathcal{L}_j)\right)\cdot\prod_{q_{1,k}+q_{2,k}=1} c_1(\mathcal{L}_k)\\
	= &w_3\cdot\prod_{i \in B^-_\stSig(v_1,v_2)} \eu((\mathcal{L}_i))\cdot
\prod_{q_{1,j}+q_{2,j}\geq1}c_1(\mathcal{L}_j). 
	\end{align*}
	Therefore,
	\beq
	\BR(\star_{V^-},\stSig)=\vr{\left.y^{v_1} y^{v_2} - y^{v_3}\cdot\prod_{i \in B^-_\stSig(v_1,v_2)} \eu(a_i\mathcal{L}_i)\cdot\prod_{j \in 
B^+_\stSig(v_1,v_2)}c_1(\mathcal{L}_j) \right| v_1,v_2\in\Boxx(\stSig) }.
	\eeq	
\end{proof}

\begin{corollary}
	For the virtual product of \cite{GLSUX:07}, we use $V=\T=\sum_{i=1}^n\mathcal{L}_i$ in the $\star_{V^-}$ product. Then
	\beq
	\Tw(\star_{virt})(v_1,v_2)=\prod_{\rho_i\subseteq\sigma(\ol{v_1}),\sigma(\ol{v_2})}c_1(\mathcal{L}_i).
	\eeq
\end{corollary}

\begin{example}	
	For the toric stack $\cX(\stSig)$ of Example~\ref{ex:p654}, consider $v_1=(0,1)$ and $v_3=(1,2)$. 
	
	Then $q_{1,1}+q_{3,1}=0+\frac12=\frac12$ and $q_{1,2}+q_{3,2}=\frac12+\frac34=\frac54$. So under the orbifold product, we have $\Tw(v_1,v_3)(\star_{orb})=c_1(\mathcal{L}_2)$. 

	On the other hand, we have $\rho_2\in\sigma(\ol{v_1})$ while $\rho_1,\rho_2\in\sigma(\ol{v_3})$. So under the virtual product, we have $\Tw(v_1,v_3)(\star_{virt})=c_1(\mathcal{L}_2)$.
\end{example}

We now give an example of a complete presentation of the inertial Chow ring
for an intertial product coming from a vector bundle on $\ix(\stSig)$.

\begin{example}\label{ex:p654inertialchowring}
	Let $\stSig$ be the stacky fan of Example~\ref{ex:p654}, with $\cX(\stSig)=\PP(6,5,4)$. As before, let $\mathcal{L}_i$ be the $G$-equivariant line bundle associated to the $i$-th component of the action of $G$ on $\cX(\stSig)$, and to the ray $\rho_i$, for $i=1,2,3$. Let $V=a_1\mathcal{L}_1+a_2\mathcal{L}_2+a_3\mathcal{L}_3$, for non-negative integers $a_1, a_2, a_3$. 	We'll compute the Chow ring for the $\star_{V^+}$ product on $\cX(\stSig)$. 
	
	As stated previously, we can calculate the ring $R_\stSig$ and its ideals $I(\stSig)$, $\CR(\stSig)$ and $\Cir(\stSig)$ without regard to the vector bundle $V$, so we'll do this first.
	
	The distinguished points are $b_1=(2,1)$, $b_2=(0,2)$ and $b_3=(-3,-4)$, so the box elements are
	\beq\begin{array}{llll}
	v_1=(0,1)&
	v_4=(-1,0)&
	v_7=(-2,-2)\quad&
	v_{10}=(0,-1)\\
	v_2=(1,1)&
	v_5=(-1,-1)&
	v_8=(-2,-3)&
	v_{11}=(1,0)\\
	v_3=(1,2)\quad&
	v_6=(-2,-1)\quad&
	v_9=(-1,-2)	
	\end{array}\eeq 
	with $q_{j,i}$ being defined by $\ol{v_j}=q_{j,1}\ol{b_1}+q_{j,2}\ol{b_2}+q_{j,3}\ol{b_3}$ as follows:
	\beq\begin{array}{llll}
	\ol{v_1}=\frac12\ol{b_2}&
	\ol{v_4}=\frac23b_2+\frac13\ol{b_3}&
	\ol{v_7}=\frac13\ol{b_2}+\frac23\ol{b_3}\quad&
	\ol{v_{10}}=\frac35\ol{b_1}+\frac25\ol{b_3}\\[1mm]
	\ol{v_2}=\frac12\ol{b_1}+\frac14\ol{b_2}&
	\ol{v_5}=\frac16\ol{b_2}+\frac13\ol{b_3}&
	\ol{v_8}=\frac15\ol{b_1}+\frac45\ol{b_3}&
	\ol{v_{11}}=\frac45\ol{b_1}+\frac15\ol{b_3}\\[1mm]
	\ol{v_3}=\frac12\ol{b_1}+\frac34\ol{b_2}\quad&
	\ol{v_6}=\frac56\ol{b_2}+\frac23\ol{b_3}\quad&
	\ol{v_9}=\frac25\ol{b_1}+\frac35\ol{b_3}	
	\end{array}
	\eeq 
	along with $\bv0=(0,0)$ (see Figure~\ref{fig:p654box} for an intuition-aiding diagram). The fan $\stSig$ has three top-dimensional cones:
	\beq
	\Boxx(\stsig_1)=\{\bv0,v_1,v_2,v_3\} \quad
	\Boxx(\stsig_2)=\{\bv0,v_1,v_4,v_5,v_6,v_7\}\quad
	\Boxx(\stsig_3)=\{\bv0,v_8,v_9,v_{10},v_{11}\}
	\eeq
	Let $w_j:=y^{v_j}$ for $j=1,\dots,11$, and let $x_i:=y^{b_i}$ for $i=1,2,3$. 
Then with coefficients in $\Z$, the stacky Stanley-Reisner ring
is 
$$\SR_{\stSig} = \Z[x_1,x_2,x_3]/\left(I(\stSig) + \Cir(\stSig)\right).$$
Since the orbifold is reduced we have $x_i = \tilde{x}_i$ for each $i$.
Based solely on the fan structure, we can compute the irrelevant ideal:
	\beq
	I(\stSig)=\vr{x_1x_2x_3}
	\eeq
	Using the distinguished points we compute the circuit ideal:
	\beq
	\Cir(\stSig)=\vr{2x_1-3x_3,x_1+2x_2-4x_3}
	\eeq	
so $$\SR_{\stSig} = \Z[x_1,x_2, x_3]/\left(2x_1 - 3x_3, x_1 + 2x_2 -4x_4, x_1x_2x_3\right)$$
Since the box has 11 elements, the ring $R_{\stSig}$ is a quotient of the ring
$$\SR_{\stSig}[w_1, w_2, \ldots , w_{10}]$$ by the ideal
$$\sum_{i=1}^{11} I_{\stSig, v_i} = \left(x_1x_3w_1, x_3w_2, x_3w_2, x_1w_4, x_1w_5, x_1w_6, x_1w_7, x_2w_8, x_2w_9,x_2w_{10}, x_2 w_{11}\right).$$
	Finally, based on the boxes of the top-dimensional cones $\sigma_1$, $\sigma_2$ and $\sigma_3$, we have the cone relations ideal:
	\begin{align*}
	\CR(\stSig)=
	\left<\begin{array}{c}w_1w_8,w_1w_9,w_1w_{10},w_1w_{11},
	w_2w_4,w_2w_5,w_2w_6,w_2w_7,w_2w_8,w_2w_9,  \\
	w_2w_{10},w_2w_{11},w_3w_4,w_3w_5,w_3w_6,w_3w_7,w_3w_8,w_3w_9,w_3w_{10},w_3w_{11}, \\
	w_4w_8,w_4w_9,w_4w_{10},w_4w_{11},w_5w_8,w_5w_9,w_5w_{10},w_5w_{11},
	w_6w_8,w_6w_9,\\
	w_6w_{10},w_6w_{11},w_7w_8,w_7w_9,w_7w_{10},w_7w_{11}\end{array}\right>
	\end{align*}
	
	This leaves only box relations ideal, which depends on $V$. By Proposition~\ref{prop:br-v-plus-minus}, we have
	\beq
	\BR(\star_{V^+},\stSig)=\vr{\left.w_{j_1}w_{j_2}-w_{j_3}\cdot\prod_{q_{1,i}+q_{2,i}\geq1}c_1(\mathcal{L}_i)^{1+a_i} \right| v_{j_1},v_{j_2}\in\Boxx(\stSig)}
	\eeq
	For each of the three top-dimensional cones $\stsig$, we must compute $w_{j_1}w_{j_2}$ for every pair $(v_{j_1},v_{j_2})\in\Boxx(\stsig)$.  (Note that we need not  calculate for box elements which are not in the same top-dimensional cone, such as $y^{v_1}y^{v_8}$, since it's already an ideal generator in $\CR(\stSig)$.)
	
	So, the generators of $\BR(\star_{V^+},\stSig)$, or at least the ones that we need to pay attention to, are
	\beq
	\begin{array}{ll}
	w_1w_2-w_3 \hspace{1.8in}&
	w_1w_3-w_2\cdot c_1(\mathcal{L}_2)^{a_2+1}\\
	w_2w_3-1\cdot c_1(\mathcal{L}_1)^{a_1+1}\cdot c_1(\mathcal{L}_2)^{a_2+1}&
	{w_1}^2-1\cdot c_1(\mathcal{L}_2)^{a_2+1}\\
	{w_2}^2-w_1\cdot c_1(\mathcal{L}_1)^{a_1+1}&
	{w_3}^2-w_1\cdot c_1(\mathcal{L}_1)^{a_1+1}\cdot c_1(\mathcal{L}_2)^{a_2+1}\\
	w_1w_4-w_5\cdot c_1(\mathcal{L}_2)^{a_2+1}&
	w_1w_5-w_4\\
	w_1w_6-w_7\cdot c_1(\mathcal{L}_2)^{a_2+1}&
	w_1w_7-w_6\\
	w_4w_5-w_6&
	w_4w_6-w_1\cdot c_1(\mathcal{L}_2)^{a_2+1}\cdot c_1(\mathcal{L}_3)^{a_3+1}\\
	w_4w_7-1\cdot c_1(\mathcal{L}_2)^{a_2+1}\cdot c_1(\mathcal{L}_3)^{a_3+1}&
	w_5w_6-1\cdot c_1(\mathcal{L}_2)^{a_2+1}\cdot c_1(\mathcal{L}_3)^{a_3+1}\\
	w_5w_7-w_1\cdot c_1(\mathcal{L}_3)^{a_3+1}&
	w_6w_7-w_5\cdot c_1(\mathcal{L}_2)^{a_2+1}\cdot c_1(\mathcal{L}_3)^{a_3+1}\\
	{w_4}^2-w_7\cdot c_1(\mathcal{L}_2)^{a_2+1}&
	{w_5}^2-w_7\\
	{w_6}^2-w_4\cdot c_1(\mathcal{L}_2)^{a_2+1}\cdot c_1(\mathcal{L}_3)^{a_3+1}&
	{w_7}^2-w_4\cdot c_1(\mathcal{L}_3)^{a_3+1}\\
	w_8w_9-w_{10}\cdot c_1(\mathcal{L}_3)^{a_3+1}&
	w_8w_{10}-w_{11}\cdot c_1(\mathcal{L}_3)^{a_3+1}\\
	w_8w_{11}-1\cdot c_1(\mathcal{L}_1)^{a_1+1}\cdot c_1(\mathcal{L}_3)^{a_3+1}&
	w_9w_{10}-1\cdot c_1(\mathcal{L}_1)^{a_1+1}\cdot c_1(\mathcal{L}_3)^{a_3+1}\\
	w_9w_{11}-w_8\cdot c_1(\mathcal{L}_1)^{a_1+1}&
	w_{10}w_{11}-w_9\cdot c_1(\mathcal{L}_1)^{a_1+1}\\
	{w_8}^2-w_9\cdot c_1(\mathcal{L}_3)^{a_3+1}&
	{w_9}^2-w_{11}\cdot c_1(\mathcal{L}_3)^{a_3+1}\\
	{w_{10}}^2-w_8\cdot c_1(\mathcal{L}_1)^{a_1+1}&
	{w_{11}}^2-w_{10}\cdot c_1(\mathcal{L}_1)^{a_1+1}
	\end{array}
	\eeq
	From these generator relations, we can eliminate $w_3, w_4, w_6, w_7$.
	
	By the associated formula of $\stSig$ (see Lemma~\ref{lem:Li-gen-char-gp}), we have the following:
	\begin{align*}
	c_1(\mathcal{L}_1)=x_1=\widetilde{x}_1=6t\\
	c_1(\mathcal{L}_2)=x_2=\widetilde{x}_2=5t\\
	c_1(\mathcal{L}_3)=x_3=\widetilde{x}_3=4t
	\end{align*}
	
	Finally, we have our ring presentation for $A^*(\cX(\stSig),\star_{V^+},\Z)$:
	\beq
	\dfr
	{\Z[t,w_1,w_2,w_5,w_8,w_{9},w_{10},w_{11}]}
	{\left<
	\begin{array}{c}
	120t^3,24t^2w_1,4tw_2, 6tw_5, 5tw_8, 5tw_9, 5tw_{10}, 
5tw_{11},\\ 
w_1w_8,w_1w_{9}, w_1w_{10},w_1w_{11},w_2w_5,w_2w_8,w_2w_{9},w_2w_{10},w_2w_{11},w_5w_8,w_5w_{9},\\
	w_5w_{10},w_5w_{11},{w_1}^2-(5t)^{a_2+1},{w_2}^2-w_1(6t)^{a_1+1},
{w_5}^4-w_1w_5(4t)^{a_3+1},\\
	{w_8}^2-w_9(4t)^{a_3+1},w_8w_9-w_{10}(4t)^{a_3+1},w_8w_{10}-w_{11}(4t)^{a_3+1},\\
	w_8w_{11}-(6t)^{a_1+1}(4t)^{a_3+1},{w_9}^2-w_{11}(4t)^{a_3+1},w_9w_{10}-(6t)^{a_1+1}(4t)^{a_3+1},\\
	{w_{10}}^2-w_8(6t)^{a_1+1},w_9w_{11}-w_8(6t)^{a_1+1},{w_{11}}^2-w_{10}(6t)^{a_1+1},w_{10}w_{11}-w_9(6t)^{a_1+1}
	\end{array}
	\right>}
	\eeq
\end{example}


\section{A new asymptotic product}\label{sec:asymptotic}
Given a vector bundle $V$ on a toric stack $\ix(\stSig)$ we show that
we can produce new
associative products on $A^*(I\ix)_\Q$ by taking the limits of the $\star_{aV^+}$
and $\star_{aV^{-}}$ prodcuts where $a$ is a positive integer going to $\infty$.

\begin{theorem}\label{thm:asymptotic-vanishing}
The following formulas define associative products on $A^*(I\ix)_\Q$.

	\benu
	\item $ y^{v_1} \star_{+\infty} y^{v_2} =0$ if $B_\stSig^+(v_1,v_2)$ is nonempty, and
	\item $ y^{v_1} \star_{+\infty} y^{v_2} = y^{v_1} \star y^{v_2} $ otherwise.
	\eenu
	\benu
	\item $ y^{v_1} \star_{-\infty} y^{v_2} =0$ if $B_\stSig^-(v_1,v_2)$ is nonempty, and
	\item $ y^{v_1} \star_{-\infty} y^{v_2} = y^{v_1} \star y^{v_2} $ otherwise. 
	\eenu
(Here $\star$ is the usual orbifold product.)
\end{theorem}

\begin{proof} Let $V = \sum_{i=1}^n {\mathcal L}_i$ and let $a$ be a
  positive integer.  The box relations ideal of
  Proposition~\ref{prop:br-v-plus-minus} for the $aV^+$ product becomes
	\begin{align*}
	\BR(\star_{aV^+},\stSig)=&
	\vr{\left.y^{v_1}y^{v_2} -y^{v_3}\cdot\prod_{i \in B^{+}_\stSig} c_1(\mathcal{L}_i))^{1+a} \right| v_1,v_2\in\Boxx(\stSig) }
\end{align*}
Now since $A^r(I\ix)_\Q =0$ for $r < \dim \ix$, the terms $c_1({\mathcal L}_i)^{1+a}$ vanish for $a$ sufficiently large. In this case the box relation
ideal becomes

$$\{y^{v_1}y^{v_2}| B_\stSig^+(v_1,v_2)\neq\emptyset\}\bigcup\{y^{v_1}y^{v_2} -y^{v_3}| B_\stSig^+(v_1,v_2)=\emptyset\} .$$
	
which gives $\star_{+ \infty}$ formula of Theorem \ref{thm:asymptotic-vanishing}.
The associativity of the $\star_{-\infty}$ product follows from a similar argument with the $\star_{aV^{-}}$product.
\end{proof}
Other associative products with fewer non-trivial products can be obtained by taking $V = {\mathcal L}_{i_1} + \ldots + {\mathcal L}_{i_m}$ where $\{i_1, \ldots i_m\}$ is a subset of $\{1, \ldots , n\}$.

\begin{example}\label{ex:p654inf}
	Consider the stacky fan $\stSig$ of Example~\ref{ex:p654inertialchowring}. We'll compute the Chow ring of $\cX(\stSig)$ with $\star_{+\infty}$ product 
Note that the calculation for the inertial Chow ring $A^*(\cX(\stSig),\star_{V^+_\infty},\Q)$ will be the same as with $A^*(\cX(\stSig),\star_{V^+},\Q)$, done
in Example \ref{ex:p654inertialchowring} except that all terms with $a_i$ as exponents vanish. So, we have
	\beq
	A^*(\cX(\stSig),\star_{+\infty},\Q) \cong
	\dfr
	{\Q[t,w_1,w_2,w_5,w_8,w_{9},w_{10},w_{11}]}
	{\left<
	\begin{array}{c}
	t^3,t^2w_1, tw_2, tw_5, tw_8, tw_9, tw_{10}, tw_{11}\\
w_1w_8,w_1w_{9},w_1w_{10},w_1w_{11},w_2w_5,w_2w_8,w_2w_{9},w_2w_{10},w_2w_{11},\\
	w_5w_8,w_5w_{9},w_5w_{10},w_5w_{11},{w_1}^2,{w_2}^2,{w_5}^4,{w_8}^2,w_8w_9,w_8w_{10},\\
	w_8w_{11},{w_9}^2,w_9w_{10},{w_{10}}^2,w_9w_{11},{w_{11}}^2,w_{10}w_{11}
	\end{array}
	\right>}
	\eeq
	
\end{example}

\bibliographystyle{amsmath}
\bibliography{refs}
\def\cprime{$'$}
\def\cprime{$'$}

\myaddress{Department of Mathematics, University of Missouri, Columbia, MO 65211}
\myemail{colemanthomasa@gmail.com}
\myemail{edidind@missouri.edu}

\end{document}